%% file: ISTIME-5.tex
\documentclass{article}

\usepackage[draft]{todonotes}
\usepackage{amsthm,amsmath}
\usepackage{amscd,amssymb,latexsym,graphicx}
\usepackage{amscd}
\usepackage[all]{xy}
\usepackage{accents}
\usepackage{cite}
\usepackage{mathtools}
\usepackage{excludeonly}
\usepackage{hycolor}
\usepackage{xcolor}
\usepackage[
                       colorlinks=true,
                       linkcolor=black,
                       citecolor=black,
                       urlcolor=blue,
                     ]{hyperref}
\input{latex-shortcuts.tex}

\title{Contraction method and Lambda-Lemma
       }
\author{Joa Weber\footnote{
        {\bf Financial support:}
        FAPESP grant 2013/20912-4, FAEPEX grant 1135/2013,
        and CNPq, Conselho Nacional de Desenvolvimento Cient\'{\i}fico
        e Tecnol\'ogico - Brasil.
  \hfill
  \newline
  joa@math.sunysb.edu
        }
        \\
        IMECC UNICAMP }
\date{\today}
\begin{document}

\maketitle 

\begin{abstract}
We reprove the $\lambda$-Lemma for
finite dimensional gradient flows by generalizing the
well-known contraction method proof of the
local (un)stable manifold theorem.
This only relies on the forward Cauchy problem. We obtain
a rather quantitative description of (un)stable foliations
which allows to equip each leaf with a copy of the flow
on the central leaf -- the local (un)stable manifold.
These dynamical thickenings are key tools in our recent work~\cite{weber:2014e}.
The present paper provides their construction.
\end{abstract}%

\tableofcontents


%
%
%


\section{Introduction and main results}\label{sec:intro}
Throughout $(M^n,g)$ denotes a Riemannian manifold of
finite dimension $n$ and $f$ is a function on $M$ of class
$C^{r+1}$ with $r\ge 1$. We assume further that the downward gradient equation
\begin{equation}\label{eq:DGE}
     \dot u(t):=\tfrac{d}{dt} u(t)
     =-(\nabla f)\circ u(t)
\end{equation}
for curves $u:[0,\infty)\to M$ generates a complete forward flow
$\varphi=\{\varphi_t\}_{t\ge 0}$. This holds true,
for instance, whenever $f$ is {\bf exhaustive}, that is whenever
all sublevel sets $M^a=\{f\le a\}$ are compact.
By $x$ we denote throughout a non-degenerate critical point of $f$
of Morse index $k=\IND(x)$ with $1\le k\le n-1$; for $k\in\{0,n\}$
there is no $\lambda$-Lemma.
In the following we give an overview of the contents of this paper.
Definitions and proofs are given in subsequent sections.

The $\lambda$-Lemma\footnote{
  The naming {\bf inclination} or {\bf\boldmath $\lambda$-Lemma}
  is due to the fact that it asserts $C^1$ convergence of a family of
  disks to a given one in the (un)stable manifold, so the
  \emph{inclination} of suitable tangent vectors converges.
  In~\cite{Palis:1969a} these inclinations were denoted by $\lambda_\ell$.
  }
was proved by Palis~\cite{palis:1967a,Palis:1969a} in the late 60's.
Its backward version asserts that, given a hyperbolic singularity $x$ of a $C^1$
vector field $X$ on a finite dimensional manifold, the corresponding backward flow
applied to any disk $D$ transverse to the unstable manifold $W^u(x)$ converges in $C^1$
and locally near $x$ to the stable manifold of $x$; cf.~\cite[Ch.~2 \S7]{palis:1982a}
and Figure~\ref{fig:fig-lambda-Lemma}.
Since the $\lambda$-Lemma is a local result we pick
convenient coordinates about $x$.

\begin{definition}[Local coordinates and notation -- Figure~\ref{fig:fig-local-setup-NEW}]
\label{def:loc-coord-lambda}
\mbox{ }

\begin{enumerate}
\item[\bf (H1)]
  Associated to the Hessian operator $A=D(\nabla f)_x$ there
  is the spectral splitting $X:=T_xM=E^-\oplus E^+$
  in~(\ref{eq:splitting}) with spectral projections $\pi_\pm$
  and where $E^{-/+}=T_xW^{u/s}(x)$.
  Consider the spectral gap $d>0$ of $A$, see~(\ref{eq:spec-A}), and fix a constant
  $\lambda\in(0,d)$ -- which will become the rate of
  exponential decay. Consider the constants
  $0<\delta<\mu<d$ defined by Figure~\ref{fig:fig-spectral-gap}.
\begin{figure}
  \centering
  \includegraphics{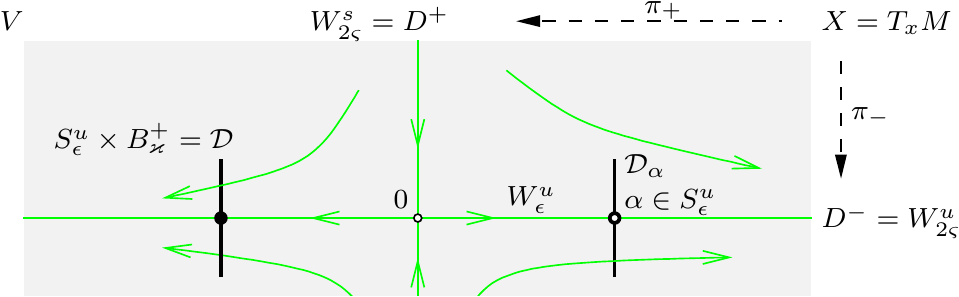}
  \caption{Local coordinates near non-degenerate critical point}
  \label{fig:fig-local-setup-NEW}
\end{figure}
\item[\bf (H2)]
  Express the downward gradient equation~(\ref{eq:DGE})
  via exponential coordinates in the form of the ODE~(\ref{eq:f})
  on a ball $B_{\rho_0}\subset X:=T_xM$ whose radius
  $\rho_0\le 1$ satisfies~(\ref{eq:rho_0}).
  In these coordinates the origin represents $x$
  and the local flow will be denoted by $\phi=\{\phi_t\}$.
  The non-linearity $h$ of this ODE, also shown in~(\ref{eq:f}),
  satisfies the Lipschitz estimates of Lemma~\ref{le:f}.

\noindent
{\bf Notation.} The notation of coordinate representatives of global
objects such as $W^s(x)\subset M$ will be the global notation
without the $x$ such as~$W^s$.
\item[\bf (H3)]
  By Remark~\ref{rem:flat-inv-mfs} we take our local model of
  the form $V=D^-\times D^+$ where $V \subset B_{\rho_0}$ and
  $D^{-/+}=W^{u/s}_{2\varsigma}\subset E^\mp$ represent a/de/scending
  disks, see~(\ref{eq:def-desc-disk}) and Remark~\ref{rem:asc-disk},
  whose parameter $2\varsigma$ is determined by
  the inclusions~(\ref{eq:F-G-inc}).
  In this local model the graph maps~(\ref{eq:graph-F}) and~(\ref{eq:graph-G})
  of the local un/stable manifolds take on the rather simple form
  $$
      \Ff^\infty=(id,0):D^-\to D^-\times D^+,\qquad
      \Gg^\infty=(0,id):D^+\to D^-\times D^+.
  $$

\item[\bf (H4)]
Suppose $\eps\in(0,\varsigma)$ and consider the descending sphere
$S^u_\eps=\p W^u_\eps$ contained in $W^u_{2\varsigma}=D^-$.
Pick a constant $\varkappa\in(0,1]$ such that the radius $\varkappa$ ball
$B_\varkappa^+\subset E^+$ lies in the ascending disk
$W^s_{\varsigma}\subset W^s_{2\varsigma}=D^+$. Consequently the hypersurface
$\Dd:=S^u_\eps\times B_\varkappa^+$ is contained in the interior of $V$.
The fiber $\Dd_{\alpha}:=\{{\alpha}\}\times B_\varkappa^+$ over ${\alpha}\in S^u_\eps$
is a translated ball of codimension~$k$.
\end{enumerate}
\end{definition}

\begin{theorem}[Backward $\lambda$-Lemma]
\label{thm:backward-lambda-Lemma}
Given the local model on $V=W^u_{2\varsigma}\times W^s_{2\varsigma}\subset E^-\oplus E^+=T_x M$
provided by Definition~\ref{def:loc-coord-lambda} nearby a
non-degenerate critical point $x$ of a function $f:M\to\R$,
pick $\eps\in(0,\varsigma)$. Then the following is true. There is a ball
$B^+\subset E^+$ about the origin of some radius $\frac{\rho}{2}>0$,
a constant $T_0>0$, and a Lipschitz continuous map
\begin{equation*}
\begin{split}
     \Gg:[T_0,\infty)\times S^u_\eps\times B^+
    &\to V\subset E^-\oplus E^+
   \\
     (T,{z_-},z_+)
    &\mapsto 
     \left(G^T_{z_-}(z_+),z_+\right)
     =:\Gg^T_{z_-}(z_+)
\end{split}
\end{equation*}
defined by~(\ref{eq:G^T}) below which is of class $C^r$ in ${z_-}$
and $z_+$. It satisfies the identity $\Gg^T_{z_-}(0)=\phi_{-T}({z_-})$ and the graph
of $G^T_{z_-}$ consists of those $z\in V$ which satisfy $\pi_+ z\in B^+$
and reach the fiber $\Dd_{z_-}=\{{z_-}\}\times B^+_\varkappa$ at time $T$, that is
\begin{equation*}
\begin{split}
     \Gg^T_{z_-}(B^+)
     ={\phi_T}^{-1}\Dd_{z_-}\cap
     \left(E^-\times B^+\right).
\end{split}
\end{equation*}
Moreover, the graph map $\Gg^T_{z_-}$ converges uniformly in $C^1$, as $T\to\infty$,
to the local stable manifold graph map $\Gg^\infty$ as illustrated by Figure~\ref{fig:fig-lambda-Lemma}.
More precisely,
\begin{equation}\label{eq:lin-lam}
\begin{split}
     \bigl\|\Gg^T_{z_-}(z_+)-\Gg^\infty(z_+)\bigr\|=\bigl\|G^T_{z_-}(z_+) \bigr\|
   &\le e^{-T\frac{\lambda}{8}}, \\
     \bigl\|d\Gg^T_{z_-}(z_+)v-d\Gg^\infty(z_+)v \bigr\|=\bigl\|dG^T_{z_-}(z_+)v \bigr\|
   &\le c_* e^{-T\frac{\lambda}{8}}\Norm{v},
\end{split}
\end{equation}
for all $T\ge T_0$, ${z_-}\in S^u_\eps$,
$z_+\in B^+$, and $v\in E^+$ where $c_*>0$ is a constant. Estimate two requires
$f\in C^{2,1}$ near $x$ in which case $T\mapsto\frac{d}{dT}\Gg(T,z_-,z_+)$
is also Lipschitz continuous.\footnote{
  The condition $f\in C^{2,1}$ near $x$ (satisfied for $r\ge2$) hinges on the
  Lipschitz Lemma~\ref{le:f}.
  }
\end{theorem}
\begin{figure}
  \centering
  \includegraphics{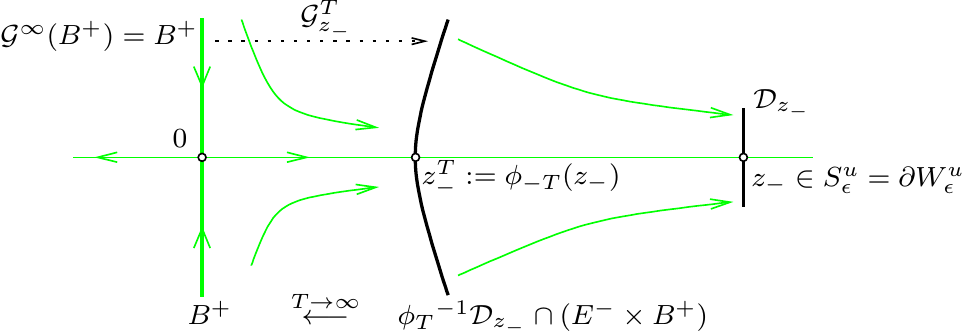}
  \caption{Backward $\lambda$-Lemma}
  \label{fig:fig-lambda-Lemma}
\end{figure}

\begin{remark}[Contraction method proof of $\lambda$-Lemma]\label{rmk:pro-con-lambda}
The contraction method proof presented below has its own dis/advantages.
On the worrying side, for $f\in C^2$, that is $X=-\nabla f\in C^1$,
we only obtain $C^0$ convergence so far, whereas we do get $C^1$
convergence if $f$ is of class $C^{2,1}$ near $x$, e.g. if $f\in C^3$.
On the bright side, we get rather useful quantitative control on each of the involved
variables such as time $t$, the variable describing the dislocated disks, and their
dependence on the base point; for details see
Theorem~\ref{thm:backward-lambda-Lemma}.
Most importantly, this quantitative control lends itself to construct foliations around
critical points, foliations associated to global objects, namely sublevel sets,
and to equip each leaf with its own semi-flow;
see~\cite{weber:2014c,weber:2014e} for first applications.
\end{remark}

\begin{remark}\label{rem:general-D}
Theorem~\ref{thm:backward-lambda-Lemma} is
a family backward $\lambda$-Lemma, namely
for the family of disks $\Dd=S^u_\eps\times B^+_\varkappa$.
Given a general hypersurface $\Dd^\prime$ whose intersection with the unstable
manifold is transverse and equal to $S^u_\eps$, one needs to change
coordinates to bring $\Dd^\prime$ into the required normal form
$S^u_\eps\times B^+_\varkappa$. To achieve this apply the parametrized Implicit Function Theorem
to construct a $C^r$ diffeomorphism of $V$ which is the identity outside a small
neighborhood $U$ of the \emph{compact} intersection locus $S^u_\eps$ and maps the
part of $\Dd^\prime$ near $S^u_\eps$ to $S^u_\eps\times B^+_\varkappa$
for some $\varkappa$.
\end{remark}

We substituted the backward flow of the disk $\Dd_{z_-}$
by its pre-image ${\phi_T}^{-1}\Dd_{z_-}$, because there are
interesting situations, cf.~\cite{henry:1981a}, with only a forward
semi-flow. For instance, the heat flow on the loop space of a closed Riemannian
manifold~\cite{salamon:2006a} is just a forward \emph{semi}-flow.
Surprisingly, although there is no backward flow, there is a
backward $\lambda$-Lemma\cite{weber:2014a} under which the hypersurface $\Dd$
moves backward in time and foliates an open set.
Furthermore, it is even still possible to construct a non-trivial Morse
complex~\cite{weber:2013b,weber:2014c}.
Consequently the -- in finite dimensions -- common situation
of a genuine forward and backward flow with stable and
unstable manifolds, all being finite dimensional, allows for (too) many choices.
In order to emphasize the necessary elements and, at
the same time, to introduce the reader tacitly to the infinite dimensional
heat flow scenario, we subject ourselves to the following convention.

\begin{convention}[No backward Cauchy problem]\label{con:back-flow}
We shall use existence of a solution to the
{\bf Cauchy problem only in forward time} and ignore the
backward Cauchy problem alltogether. However, \emph{on the unstable
manifold} there is still a natural procedure to define a backward flow;
see Definition~\ref{def:alg-back-flow}. Since there is no Cauchy problem involved,
we will (and need to) use this so-called {\bf algebraic backward
flow}; this is coherent with the infinite dimensional
case cited above. In finite dimensions the algebraic
backward flow coincides
with the one obtained via the Cauchy problem.
  So we do use a backward flow, but without solving the
  backward Cauchy problem and only along
  the unstable manifold.
\\
To put things positively, our self-imposed lack of mathematical
structure eliminates the possibility of choices and
therefore reveals those elements which are essential to define
a Morse complex -- or even just an unstable manifold.
\end{convention}

The simple but far reaching idea, see~\cite{weber:2014a,weber:2014c},
which avoids any backward Cauchy
problem is to look at the \emph{pre-image} under the time-$T$-map
of a given disk $\Dd_q(x)$ transverse at $q$ to the unstable manifold $W^u(x)$.
The obvious but non-trivial problem is then to show that these
\emph{sets} are in fact \emph{submanifolds} of $M$. To show this we shall
write them locally near $x$ as graphs over the stable tangent space
$E^+=T_x W^s(x)$ using the {\bf contraction method} which in
this case consists of interpreting the backward Cauchy problem as
a ``mixed Cauchy problem''. The latter can be reformulated in terms of
the fixed point of a contraction, both depending on parameters.
Thereby we obtain a new proof of the $\lambda$-Lemma in finite dimensions;
see~\cite{weber:2014a} for the infinite dimensional semi-flow case of a parabolic PDE
where this program has been carried out first.
A key point of interest will be to control how the pre-images depend on time $T$.

\subsubsection*{Stable foliations}
Borrowing from~\cite{salamon:1990a}
we define for each choice of constants $\eps,\tau>0$ a pair
\begin{equation}\label{eq:Conley-set-NEW}
\begin{split}
     N_x^s=N_x^s(\eps,\tau)
   :&=
     \left\{p\in M\mid\text{$f(p)\le c+\eps$,
     $f(\varphi_\tau p)\ge c-\eps$}\right\}_x ,
\end{split}
\end{equation}
where $\{\ldots\}_x$ denotes the
{\bf path connected component}
that contains $x$, and
\begin{equation}\label{eq:L_x}
     L_x^s=L_x^s(\eps,\tau)
     :=\{p\in N_x^s\mid f(\varphi_{2\tau} p)\le c-\eps\}.
\end{equation}
Figure~\ref{fig:fig-Conley-pair} shows a typical
\begin{figure}
  \centering
  \includegraphics{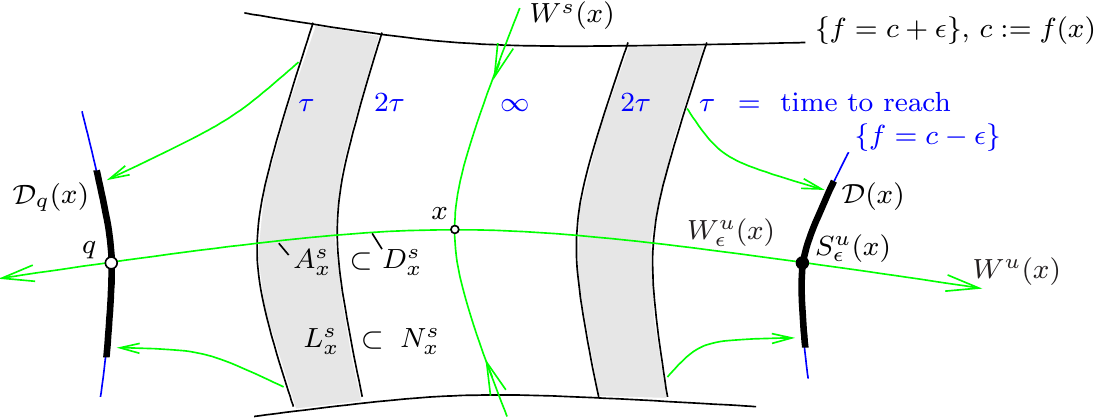}
  \caption{Conley pair $(N_x^s,L_x^s)$ associated to critical point $x$
           }
  \label{fig:fig-Conley-pair}
\end{figure}
such pair, illustrates what is called the exit set property
of $L_x^s$, and indicates hypersurfaces which are characterized by the
fact that each point reaches the level set $\{f=c-\eps\}$ in the same time.
The points on the stable manifold never reach level $c-\eps$, so they are
assigned the time label $\infty$. 
By the backward $\lambda$-Lemma,
Theorem~\ref{thm:backward-lambda-Lemma},
these hypersurfaces foliate some neighborhood of $x$.
However, the neighborhood and so the leaves have no global meaning so far.
It is the content of Theorem~\ref{thm:inv-fol} that the pairs
$(N_x^s,L_x^s)$ which are defined in terms of level sets are foliations;
for a first application see~\cite{weber:2014e}.

\begin{remark}[Forward $\lambda$-Lemma and unstable foliations]\label{rem:forward-unstable-case}
Since the backward $\lambda$-Lemma as well as the stable foliations
$(N_x^s,L_x^s)$ with respect to $f$ are of local nature at $x$ one can simply obtain
the forward version of Theorem~\ref{thm:backward-lambda-Lemma}
and the unstable foliations $(N_x^u,L_x^u)$ with respect to $f$
by taking the former ones with respect to $-f$.
(If necessary, cut off $-f$ to obtain a complete forward flow.)
\end{remark}

\section{Preliminaries}\label{sec:setup-hyp-sing}

Throughout we denote by $\nabla=\nabla^g$ the Levi-Civita connection associated to
the  Riemannian metric $g$ on $M$ which we also denote by $\langle\cdot ,\cdot\rangle$.
Given a singularity $p$ of a vector field $X$ on $M$ we denote by $DX_p$
linearization $dX_p:T_pM\to T_{X(p)}TM\cong T_pM\oplus T_pM$
followed by projection onto the second component. The isomorphism is natural, because $X(p)=0$.
There is the identity $DX_p=(\Nabla{\cdot} X)_p$.

\subsubsection*{Hessian and spectrum}
By $\Crit f$ we denote the set of critical point of $f$.
Given $p\in\Crit f$, in other words a singularity of $X=\nabla f$,
then there is a linear symmetric operator $A=A_p:T_pM\to T_pM$
uniquely determined by the identity
\begin{equation}\label{eq:Hess-op}
     \Hess_p f\, (v,w)
     =\langle A v,w\rangle
\end{equation}
for all $v,w\in T_pM$. The bilinear form $\Hess_p f$ is defined in local
coordinates in terms of the (symmetric) matrix with entries
$(\p^2 f/\p x^i\p x^j)(p)$. It holds that $Av=D(\nabla f)_pv=(\Nabla{\tilde v}\nabla f)_p$
where $\tilde v$ is a vector field on $M$ with $\tilde v_p=v$.
By symmetry the spectrum of the {\bf Hessian operator} $A$ is real. If it does not contain
zero, then $p$ is called a {\bf non-degenerate} critical point of $f$.

Suppose $x$ is a non-degenerate critical point of $f$.
Then by the spectral theorem for symmetric operators the spectrum of
$A=A_x$ is given by $n$ reals
\begin{equation}\label{eq:spec-A}
     \lambda_1\le\dots\le\lambda_k<0
     <\lambda_{k+1}\le\dots\le\lambda_n
\end{equation}
counting multiplicities.
The distance $d>0$ between $0$ and
the spectrum of $A$ is called the
{\bf spectral gap} of $A$.
The number $k$ of negative eigenvalues
is called the {\bf Morse index} of the critical
point $x$ and denoted by $\IND(x)$.
Denote by $E^-$ and $E^+$ the span of all
eigenvectors associated to negative and
positive eigenvalues, respectively.
Then there is the orthogonal splitting
\begin{equation}\label{eq:splitting}
     T_xM=E^-\oplus E^+
\end{equation}
with associated
{\bf orthogonal projections} $\pi_-:T_xM\to E^-$
and $\pi_+:T_xM\to E^+$.
Since $A$ preserves $E^-$ and
$E^+$ it restricts to linear operators
$A^-:=A\circ\pi_-$ and $A^+:=A\circ\pi_+$ on $E^-$
and $E^+$, respectively.
If $0<k<n$, then $\lambda_1<0<\lambda_n$.

\subsubsection*{Linearized equation and linearized flow}
Linearizing the downward gradient equation~(\ref{eq:DGE})
at a constant solution $u(t)\equiv p\in \Crit f$ in direction
of a vector field $\xi:\R\to T_pM$ along $u$ leads to the linear autonomous~ODE
\begin{equation}\label{eq:lin-eq}
     \dot\xi(t)=-A\xi(t)
\end{equation}
for $\xi$ where $A=A_p$ is the Hessian operator determined by~(\ref{eq:Hess-op}).
 
Interpreting $-A:T_p M\to T_pM$ in~(\ref{eq:lin-eq})
as a vector field on the manifold $T_pM$
it is easy to see that the flow generated by $-A$
is the 1-parameter family $\{\Phi_t\}_{t\in\R}$
of invertible linear operators on $T_pM$ given by
\begin{equation}\label{eq:e^tA}
     \Phi_t =e^{-tA}:=\sum_{j=0}^\infty\frac{(-t)^j}{j!}A^j.
\end{equation}
If the spectrum of $A$ is given by $\{\lambda_1,
\dots,\lambda_n\}\subset\R$, then $\{e^{-t\lambda_1},
\dots,e^{-t\lambda_n}\}\subset(0,\infty)$ is the
spectrum of $e^{-tA}$.
We will use $\Phi_t=e^{-tA}=e^{-tA^-}\oplus e^{-tA^+}$ only for
$t\ge 0$. However, the part $e^{-tA^-}$ will be needed for all
$t\in\R$. Since this part acts on the (linear) unstable manifold $E^-$
and is defined by a matrix exponential -- thus without
employing any backward Cauchy problem --
there is no conflict with Convention~\ref{con:back-flow}.

\begin{lemma}[Linearized flow satisfies linearized equation]
Given $p\in\Crit f$, then the family of linearizations $\{d\varphi_t(p)\}_{t\ge 0}$
coincides with the time-$t$-maps $\{\Phi_t\}_{t\ge 0}$ associated to the linearized
vector field $-A=-D(\nabla f)_p$ on $T_pM$.
\end{lemma}

\begin{proof}
See e.g.~\cite[Le.~2.5]{weber:2006b}.
\end{proof}

\begin{proposition}[Exponential estimates]\label{prop:exp-est}
Consider the Hessian operator $A=A^-\oplus A^+$
associated to a non-degenerate critical point
$x$ of Morse index $k$ and with induced splitting
$T_xM=E^-\oplus E^+$ provided by~(\ref{eq:splitting}). Fix a constant $\mu$
in the spectral gap $[0,d]$ of $A$. In this case
there are the estimates
$$
     \bigl\| e^{-tA^+}\bigr\|
     \le e^{-t\lambda_{k+1}}
     \le e^{-t\mu}
     ,\qquad
     \bigl\| e^{-tA^-}\bigr\|
     \le e^{-t\lambda_1},
$$
for every $t\ge 0$ and
$$
     \bigl\| e^{-tA^+}\bigr\|
     \le e^{-t\lambda_n}
     ,\qquad
     \bigl\| e^{-tA^-}\bigr\|
     \le e^{-t\lambda_k}
     \le e^{t\mu},
$$
for every $t\le 0$.
Thus
$
     \norm{e^{-tA}}
     \le\max\{e^{-t\lambda_{k+1}}, e^{-t\lambda_1}\}
     \le\max\{e^{-t\mu}, e^{-t\lambda_1}\}
$
whenever $t\ge 0$ and
$
     \norm{e^{-tA}}
     \le\max\{e^{-t\lambda_n}, e^{-t\lambda_k}\}
     \le\max\{e^{-t\lambda_n}, e^{t\mu}\}
$
whenever $t\le 0$.
\end{proposition}

\begin{proof}
Eigenvectors associated to $\lambda_j$ and $e^{-t\lambda_j}$ can be chosen equal.
\end{proof}

\subsubsection*{Exponential map}
Given any point $q$ of our Riemannian manifold $M$, we denote by $\exp_q:T_qM\supset \Oo_{\iota_q}\to M$
the associated exponential map. Here $\Oo_{\iota_q}$ is the open ball whose radius $\iota_q>0$ is the injectivity
radius of $\exp$ at the point $q$. The infimum $\iota\ge 0$ of $\iota_q$ over $M$ is called the
{\bf injectivity radius of \boldmath$M$}. The map
$$
     \exp:TM\supset\Oo_\iota\to M
$$
is called the {\bf exponential map} of the Riemannian manifold.
It is a piece of mathematical folklore that one can introduce
global maps $E_1$ and $E_2$ which can be viewed as partial derivatives
of the exponential map in the direction of $M$ and in the fiber direction, respectively.
Also derivatives $E_{ij}$ of second and higher order can be introduced; see e.g.~\cite[\S2.1]{weber:2014a}.

\begin{theorem}\label{thm:exp-map}
Given $u\in M$ and $\xi\in\Oo_{\iota_u}\subset 
T_uM$, then there are linear maps
$$
  E_i (u,\xi):T_uM\to T_{exp_u\xi}M ,\qquad
  E_{ij}(u,\xi):T_uM\times T_uM\to T_{exp_u\xi}M,
$$
for $i,j\in\{1,2\}$ such that
the following is true. If $u:\R\to M$ is a smooth curve
and $\xi,\eta$ are smooth vector fields along $u$
such that $\xi(s)\in\Oo_{\iota_{u(s)}}$ for every $s$,
then the maps $E_i$ and $E_{ij}$ are
{\bf characterized}\footnote{uniquely determined.}
by the identities
\begin{equation} \label{eq:exponential-identity}
\begin{split}
     \frac{d}{ds}\exp_u(\xi)
    &=E_1(u,\xi)\p_su
     +E_2(u,\xi)\Nabla{s}\xi
    \\
     \Nabla{s}\left( E_1(u,\xi)\eta\right)
    &=E_{11}(u,\xi)\left(\eta,\p_s u\right)
      +E_{12}(u,\xi)\left(\eta,\Nabla{s}\xi\right)
      +E_1(u,\xi)\Nabla{s}\eta
    \\
     \Nabla{s}\left( E_2(u,\xi)\eta\right)
    &=E_{21}(u,\xi)\left(\eta,\p_s u\right)
      +E_{22}(u,\xi)\left(\eta,\Nabla{s}\xi\right)
      +E_2(u,\xi)\Nabla{s}\eta.
\end{split}
\end{equation}
These maps satisfy the identities
\begin{equation}\label{eq:E_ij(0)}
     E_1(x,0)=E_2(x,0)=\1
     ,\quad
     E_{11}(x,0)=E_{21}(x,0)=E_{22}(x,0)=0.
\end{equation}
\end{theorem}

\subsubsection*{Tangent coordinates near critical point}
\begin{lemma}\label{le:conv-coord}
Pick $p\in Crit f$ and let $\iota_p>0$ be its
injectivity radius. Then in the neighborhood
$\exp_p\Oo_{\iota_p}$ of $p$ in $M$ the solutions
to $\dot u=-(\nabla f)\circ u$ correspond
via the identity $u(t)=\exp_p\xi(t)$ precisely to
the solutions of the ODE
\begin{equation}\label{eq:f}
     \dot\xi(t)+A\xi(t)=h(\xi(t))
     ,\qquad
     h(\xi)=-E_2(p,\xi)^{-1}(\nabla f)_{\exp_p\xi}
     +A\xi,
\end{equation}
where $\xi$ takes values in the open ball
$\Oo_{\iota_p}\subset T_p M$
and $h(\xi)$ abbreviates $h(\xi(t))$.
\end{lemma}

\begin{proof}
Given any smooth curve $\xi$ in $\Oo_{\iota_p}$, set
$u(t)=\exp_p\xi(t)$. Take the derivative
to get that $\dot u=E_2(p,\xi)\dot\xi$
pointwise in $t$. Thus we obtain that
$$
     E_2(p,\xi)\dot\xi=\dot u
     =-(\nabla f)\circ u
     =-E_2(p,\xi)A\xi
     -\left((\nabla f)_{\exp_p\xi}-E_2(p,\xi)A\xi\right)
$$
pointwise in $t$ and where the last step is by
adding zero.
\end{proof}

\subsubsection*{Lipschitz continuity of the non-linearity}
\begin{lemma}[Lipschitz near $p\in\Crit f$]
\label{le:f}
   Recall that $f\in C^{r+1}(M,\R)$ for $r\ge 1$. Fix $p\in\Crit f$
   and let $\iota_p$ be its injectivity radius. Set $\rho_0:=\iota_p/4$. Then there is
   a continuous and non-decreasing function $\kappa:[0,\rho_0]\to[0,\infty)$
   vanishing at $0$ such that the following is true.
   The non-linearity $h:B_{\rho_0}\to T_pM$ given by~(\ref{eq:f}) is
   of class $C^r$. It satisfies $h(0)=0$ and $dh(0)=0$ and the Lipschitz estimate
   \begin{equation*}
     \Norm{h(\xi)-h(\eta)}
     \le \kappa(\rho)\Norm{\xi-\eta}
   \end{equation*}
   and its consequence
   \begin{equation}\label{cor:f}
     \Norm{dh(\xi)v}
     \le \kappa(\rho)\Norm{v}
   \end{equation}
   whenever $\norm{\xi},\norm{\eta}\le\rho\le\rho_0$ and $v\in T_pM$.
   If $f$ is of class $C^{2,1}$ locally near $p$,
   then there is a constant $\kappa_*>0$ such that $dh$ satisfies the Lipschitz estimate
   \begin{equation*}
   \begin{split}
     \Norm{dh(\xi)v-dh(\eta)v}
    &\le \kappa_*
     \Norm{\xi-\eta}
     \Norm{v},
   \end{split}
   \end{equation*}
   whenever $\norm{\xi},\norm{\eta} \le\rho\le\rho_0$ and $v\in T_pM$.
\end{lemma}

\begin{proof}[Proof of the Lipschitz Lemma~\ref{le:f}]
That $f\in C^{r+1}$ implies $h\in C^r$ follows immediately from the
definition of $h$; see~(\ref{eq:f}).
Pick $\xi,\eta\in B_{\rho_0}$ and set $X:=\eta-\xi$.
It is useful to view $\xi$ as being fixed and 
$\eta(X)=\xi+X$ as depending on $X$.
Consider the map defined by
\begin{equation*}
\begin{split}
     I(X)
    :=h(\xi)-h(\eta)
   &=-E_2(p,\xi)^{-1}(\nabla f)_{\exp_p\xi}+A\xi\\
   &\quad
     +E_2(p,\xi+X)^{-1}(\nabla f)_{\exp_p(\xi+X)}-A(\xi+X)
\end{split}
\end{equation*}
and note that $I(0)=0$. Moreover, for any
$v\in T_pM$ it holds that
\begin{equation*}
\begin{split}
     dI(X)v
   &=\left.\tfrac{D}{d\tau}\right|_0 I(X+\tau v)\\
   &=-E_2(p,\xi+X)^{-1}E_{22}(p,\xi+X)
     \left[(\nabla f)_{\exp_p(\xi+X)},E_2(p,\xi+X)^{-1}v\right]
     \\
   &\quad
     +E_2(p,\xi+X)^{-1} D(\nabla f)_{\exp_p(\xi+X)}
     E_2(p,\xi+X) v -Av 
\end{split}
\end{equation*}
where brackets $[\dots]$ indicate (multi)linearity.
Thus we obtain that
\begin{equation*}
\begin{split}
     \norm{h(\xi)-h(\eta)}
   &=\norm{I(X)}
     =\norm{I(X)-I(0)}
     =\norm{dI(\sigma X)X}\\
   &\le
     \norm{{E_2}^{-1}}_{L^\infty(B_{3\rho_0})}^2
     \norm{E_{22}}_{L^\infty(B_{3\rho})}
     \norm{(\nabla f)_{\exp_p}}_{L^\infty(B_{3\rho})}
     \norm{X}\\
   &\quad
     +\norm{{E_2}^{-1} D(\nabla f)_{\exp_p}
     E_2-A}_{L^\infty(B_{3\rho})}\norm{X}\\
   &=:\kappa(\rho) \norm{\xi-\eta}
\end{split}
\end{equation*}
where for instance $E_2$ abbreviates $E_2(p,\cdot)$
and $L^\infty$ denotes the sup-norm.
Moreover, existence of some constant
$\sigma\in[0,1]$ is asserted by Taylor's theorem.
That the function $\kappa(\rho)$ is non-decreasing
is due to the fact that the supremum is taken over
the closed ball of radius $3\rho$. Indeed the
supremum is obviously non-decreasing if it is taken
over larger balls, that is if $\rho$ grows.
For $\rho=0$ all maps are evaluated at the origin
of $T_pM$, thus $\kappa(0)=0$ since\footnote{
  In fact, we don't even need to use the identity
  $E_{22}(p,0)=0$ since anyway $(\nabla f)_p=0$.
  }
$E_{22}(p,0)=0$ and $E_2(p,0)=\1$
by~(\ref{eq:E_ij(0)}) and since $D(\nabla f)_p-A=0$.
This proves the Lipschitz estimate. Concerning its consequence recall that
$dh(\xi)v=\lim_{\tau\to  0}\frac{h(\xi+\tau v)-h(\xi)}{\tau}$
and apply the Lipschitz estimate.

Concerning the Lipschitz estimate for $dh$ note that
by~(\ref{eq:f}) we get that
\begin{equation*}
\begin{split}
     dh(\xi)v=\left.\tfrac{D}{d\tau}\right|_0 h(\xi+\tau v)
   &=E_2(p,\xi)^{-1}E_{22}(p,\xi)
     \left[(\nabla f)_{\exp_p(\xi)},E_2(p,\xi)^{-1}v\right]
     \\
   &\quad
     -E_2(p,\xi)^{-1} D(\nabla f)_{\exp_p(\xi)} E_2(p,\xi) v +Av.
\end{split}
\end{equation*}
Since by assumption $D(\nabla f)_{\exp_p(\cdot)}$ is locally Lipschitz near the origin, so is $dh$.
\end{proof}

\subsubsection*{Cauchy problem and integral equation}
\begin{proposition}\label{prop:representation-formula}
    Given a non-degenerate critical point $x\in\Crit f$, let $\iota_x$ be
    its injectivity radius and consider the ODE~(\ref{eq:f})
    in $T_xM$ with non-linearity $h$. Fix a radius
    \begin{equation}\label{eq:rho_0}
        \rho_0\in\left(0,\min\{1,\iota_x/4\}\right]
    \end{equation}
    sufficiently small such that the image
    $\exp_x B_{\rho_0}\subset M$ of the closed
    ball $B_{\rho_0}\subset T_xM$
    contains no critical point of $f$ other than $x$.
    Pick $T\ge0$ and assume that
    $\xi:[0,T]\to T_xM$ is a map bounded by $\rho_0$.
    Then the following are equivalent. 
\begin{enumerate}
  \item[\rm (a)] 
    The map $\xi:[0,T]\to T_xM$ of class $C^1$ is the
    (unique) solution of the Cauchy problem given by
    the localized downward gradient flow
    equation~(\ref{eq:f}) with initial value $\xi(0)$.

  \item[\rm (b)]
    The map $\xi:[0,T]\to T_xM$ is
    continuous and satisfies the integral equation
    or {\bf representation formula}
    \begin{equation}\label{eq:representation-formula}
    \begin{split}
      \xi(t) 
     &=e^{-tA}\pi_+ \xi(0)
      +\int_0^t e^{-(t-\sigma)A}
      \pi_+h(\xi(\sigma))\, d\sigma
      \\
     &\quad
      +e^{-(t-T)A^-}\pi_- \xi(T)
      -\int_t^T e^{-(t-\sigma)A^-}
      \pi_-h(\xi(\sigma))\, d\sigma
    \end{split}
    \end{equation}
    for every $t\in[0,T]$.
    In the limit $T\to\infty$ term three in the sum
    disappears.
\end{enumerate}
\end{proposition}

\begin{proof}
Essentially variation of constants; cf.~\cite[\S 9.2]{teschl:2012a} or~\cite[\S 7.3]{jost:2011a}.
\end{proof}

\section{Invariant manifolds}\label{sec:inv-mfs}
Suppose $x\in M$ is a hyperbolic singularity of the
vector field $X=-\nabla f$, that is a non-degenerate critical point of $f$.
The {\bf stable manifold} of the flow-invariant set $\{x\}$ is defined by
\begin{equation}\label{eq:st-mf}
     W^s(x)
     :=\bigl\{q\in M\,\big|\,\text{$\lim_{t\to\infty}\varphi_tq$
     exists and is equal to $x$}\bigr\}.
\end{equation}
In case of a genuine complete flow one simply considers the limit $t\to-\infty$
to define the {\bf unstable manifold} $W^u(x)$.
Note that -- despite the naming -- these sets are, at this stage, nothing but  sets.
They are invariant under the flow though.
A common strategy to endow them with a differentiable
structure is to represent them locally near $x$ as
graphs of differentiable maps and then use the flow
and flow invariance to transport the resulting coordinate
charts to any location on the un/stable manifold.
To carry out the graph construction one introduces in an
intermediate step what is called \emph{local} un/stable manifolds.


\subsection*{Unstable manifold theorem}
In view of our Convention~\ref{con:back-flow} to ignore the backward
Cauchy problem, already defining $W^u(x)$ by the analogue
of~(\ref{eq:st-mf}) is not possible. A way out is to consider an asymptotic
boundary value problem instead: Consider the set of
all forward semi-flow lines that emanate at time $-\infty$ from the non-degenerate
critical point $x$, then evaluate each such solution at time zero and
define $W^u(x)$ to be the set of all these evaluations. In symbols,
the {\bf unstable manifold} of $x$ is defined by
\begin{equation}\label {eq:unst-mf}
     W^u(x)
     :=\bigl\{u(0)\,\big|\,\text{$u:(-\infty,0]\to M$,
     (\ref{eq:DGE}),
     $\lim_{t\to-\infty} u(t)=x$}\bigr\}.
\end{equation}
It is non-empty since the constant trajectory $u\equiv x$
contributes the element $x$.

\begin{theorem}[Unstable manifold theorem]\label{thm:unst-mf-thm}
Non-degeneracy of $x$ together with $\varphi$ being a gradient flow
implies that the unstable manifold $W^u(x)$
is an embedded submanifold of $M$ of class $C^r$
tangent at $x$ to the vector subspace $E^-\subset T_xM$
of dimension $k=\IND(x)$ and diffeomorphic to $E^-$.
\end{theorem}

\begin{corollary}[Descending disks]\label{cor:desc-disk}
Given $x\in\Crit f$ non-degenerate, then
there is a constant $\eps_x>0$ such that the following is true.
\begin{enumerate}
\item[{\rm a)}]
  Each {\bf descending disk} defined by
  \begin{equation}\label{eq:def-desc-disk}
     W^u_\eps(x):=W^u(x)\cap\{f\ge f(x)-\eps\},\qquad
     \eps\in(0,\eps_x],
  \end{equation}
  is $C^r$ diffeomorphic, as a manifold-with-boundary,
  to the closed unit disk $\D^k\subset\R^k$ where $k=\IND(x)$.
  The boundary
  $$
     S^u_\eps(x)=W^u(x)\cap\{f= f(x)-\eps\}
  $$
  of a descending disk is called a {\bf descending sphere}.
\item[{\rm b)}]
  Each open neighborhood of $x$ in $M$,
  thus each open neighborhood of $x$ in $W^u(x)$,
  contains a descending disk.
\end{enumerate}
\end{corollary}

\begin{proof}[Suggestion for proof.]
Apply Theorem~\ref {thm:unst-mf-thm} and the version
of the Transversality Theorem for manifolds-with-boundary, see~\cite[Ch.~1 \S 4]{hirsch:1976a},
to obtain the $C^r$ manifold structure of the descending disk.
Use the Morse Lemma~\cite{hirsch:1976a}, which causes loss
of regularity, only to control the locus.
\end{proof}

The proof of the unstable manifold Theorem~\ref{thm:unst-mf-thm}
is a Corollary of the local unstable manifold
Theorem~\ref{thm:loc-unst-mf-thm} below. The standard argument is
to use the forward flow to move the coordinate charts provided by
Theorem~\ref{thm:loc-unst-mf-thm} near $x$ to any point of $W^u(x)$.
This shows that $W^u(x)$ is injectively immersed. Now exploit the
gradient flow property. To prove Theorem~\ref{thm:loc-unst-mf-thm}
we need a backward flow, but only on the unstable
manifold, which is coherent with Convention~\ref{con:back-flow}.

\begin{definition}[{Algebraic\footnote{
  The wording ``algebraic'' backward flow is
  only meant to indicate that no backward Cauchy problem
  is involved in its definition. It arises naturally along the
  unstable manifold each of whose points has a past by definition.
  Thus along $W^u(x)$ it turns into a genuine flow.
  }} backward flow on unstable manifold]
\label{def:alg-back-flow}
Given $u(0)\in W^u(x)$, set $q:=u(0)$ and define
$\psi_{-t} q:=u(-t)$ for $t\ge 0$. Note that
$\psi_{-t}\psi_{-s} =\psi_{-t-s}$ and that $\psi_{-t} q$
solves the backward time Cauchy
problem~(\ref{eq:DGE}) with initial value $q$.
Therefore it doesn't violate
Convention~\ref{con:back-flow} if, for any $q\in W^u(x)$,
we use the notation $\varphi_tq$ for any time $t$, positive or negative.
\end{definition}

While $W^u(x)$ is obviously backward and forward
flow invariant, a descending disk still is backward invariant
since its boundary lies in a level set.

\subsubsection*{Local unstable manifold theorem}
\label{sec:loc-unst-mf}
We wish to prove that $W^u(x)$ carries locally near $x$ the structure of a manifold. 

\begin{figure}
\hfill
\begin{minipage}[b]{.41\linewidth}
  \centering
  \includegraphics
                             {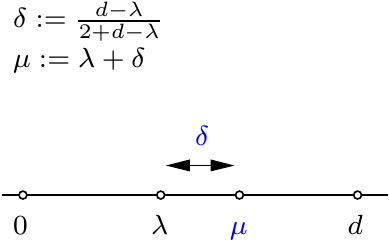}
  \caption{Spectral gap $d$} 
  \label{fig:fig-spectral-gap}
\end{minipage}
\hfill
\begin{minipage}[b]{.57\linewidth}
  \centering
  \includegraphics
                             {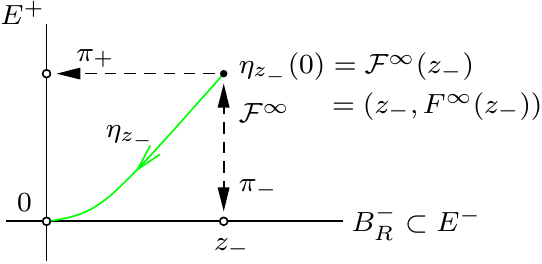}
  \caption{Graph map $\Ff^\infty:B^-_R\to T_xM$}
  \label{fig:fig-graph-map-F}
\end{minipage}
\hfill
\end{figure}

\begin{theorem}[Local unstable manifold,
Hadamard-Perron~\cite{hadamard:1901a,perron:1928a}]
\label{thm:loc-unst-mf-thm}
Assuming~(H1--H2) in Definition~\ref{def:loc-coord-lambda}
there is a constant $\rho=\rho(\lambda)\in(0,\frac{\rho_0}{2})$ such that
the following is true. A neighborhood of $0$ in the {\bf local unstable manifold}
\begin{equation}\label{eq:loc-unstab-mf}
     W^u(0,B_{\rho_0})
     :=\Bigl\{\eta(0)\,\big|\,\text{
     $\eta:(-\infty,0]\to B_{\rho_0}$,
     (\ref{eq:f}),
     $\lim_{t\to-\infty} \eta(t)=0$}\Bigr\}
\end{equation}
is a graph over the radius $R=\frac{\rho}{2}$ ball $B^-_R\subset E^-$ and
this graph is tangent to $E^-$ at $0$. More precisely, there is a $C^r$ map
\begin{equation}\label{eq:graph-F}
     \Ff^\infty
     =\left(id,F^\infty\right)
     :B^-_R\to E^-\oplus E^+
     ,\quad
     F^\infty(0)=0
     ,\quad
     dF^\infty(0)=0,
\end{equation}
whose image is a neighborhood of $0$ in $W^u(0,B_{\rho_0})$. Moreover, it holds that
$
     \Norm{\phi_t\eta_0}\le \rho e^{t\lambda}
$
for every $t\le 0$ and uniformly in $\eta_0\in\Ff^\infty(B^-_R)$;
see Figure~\ref{fig:fig-graph-map-F}.\footnote{
  {\bf Uniform exponential decay:}
  The theorem shows that all backward trajectories which remain forever in
  backward time in a certain neighborhood of the fixed point $0$ not only
  converge to $0$, but they do so exponentially -- even uniformly at the same rate of decay.
  }
\end{theorem}

We sketch the proof of the theorem by the {\bf contraction method}.
Pick an element $z_-\in E^-$ near the origin. Our object of interest is a backward flow line
$\eta:(-\infty,0]\to B_{\rho_0}$ whose value at time zero
projects to $z_-$ under $\pi_-$ and which emanates from the origin asymptotically at time
$t=-\infty$. For any sufficiently small constant $\rho\in(0,\frac{\rho_0}{2})$ the complete metric space
\begin{equation*}
\begin{split}
     Z^-=Z^-_{\lambda,\rho}
     :=\Bigl\{\eta\in  
     C^0((-\infty,0],T_xM)\,\,\Big|\,\,
     \Norm{\eta}_{\exp}^-
     :=\sup_{t\le 0} e^{-t\lambda}
     \Norm{\eta(t)}
     \le\rho\Bigr\}
\end{split}
\end{equation*}
carries the contraction given by
\begin{equation}\label{eq:Phi}
\begin{split}
     \left(\Phi_{z_-}\eta\right)(s)
     :=e^{-sA^-} z_-
   &-\int_s^0 e^{-(s-\sigma)A^-}\pi_-
     h(\eta(\sigma)) d\sigma
   \\
   &+\int_{-\infty}^s e^{-(s-\sigma)A}\pi_+
     h(\eta(\sigma)) d\sigma.
\end{split}
\end{equation}
Its (unique) fixed point $\eta_{z_-}$ is the desired flow line. Define the graph map by
$F^\infty(z_-):=\pi_+(\eta_{z_-}(0))$ as illustrated by Figure~\ref{fig:fig-graph-map-F} and denote by
$B^-_R\subset E^-$ the ball of radius $R=\frac{\rho}{2}<\frac{\rho_0}{4}$ about $0$.

\subsection*{Stable manifold theorem}
Whereas the stable manifold is easier to define -- given only a forward
flow -- than the unstable manifold, the step from local to global
is not obvious any more, given Convention~\ref{con:back-flow}. 
We shall use this oportunity to promote
Henry's~\cite{henry:1981a}, widely unkown as it seems,
argument to pull \emph{back} the local coordinate charts near $x$
in the backward time direction utilizing only the \emph{forward} flow.

\begin{theorem}[Stable manifold theorem]\label{thm:st-mf-thm}
Non-degeneracy of $x$ together with $\varphi$ being a gradient flow
implies that the stable manifold $W^s(x)$ defined by~(\ref {eq:st-mf})
is an embedded submanifold of $M^n$ of class $C^r$
tangent at $x$ to the vector subspace $E^+\subset T_xM$.
Thus $\dim W^s(x)$ is equal to the Morse co-index $n-k$ of $x$.
\end{theorem}

\begin{remark} \label{rem:asc-disk}
The {\bf ascending disk} $W^s_\eps(x)$ and the
{\bf ascending sphere} $S^s_\eps(x)$ are defined as
in Corollary~\ref{cor:desc-disk}, just replace the superlevel set
$\{f\ge f(x)-\eps\}$ by the sublevel set $\{f\le f(x)+\eps\}$.
They also have analogous properties, in the assertions
just replace $k$ by $n-k$ where $n=\dim M$.
\end{remark}


\begin{theorem}[Local stable manifold, Hadamard-Perron~\cite{hadamard:1901a,perron:1928a}]
\label{thm:loc-st-mf-thm}
Assuming~(H1--H2) in Definition~\ref{def:loc-coord-lambda}
there is a constant $\rho=\rho(\lambda)\in(0,\frac{\rho_0}{2})$ such that
a neighborhood of $0$ in the {\bf local stable manifold}
\begin{equation}\label{eq:loc-stab-mf}
     W^s(0,B_{\rho_0})
     :=\Bigl\{ z\in B_{\rho_0}\,\big|\,
     \text{$\phi_t z\in B_{\rho_0}$ $\forall t>0$ and
     $\lim_{t\to\infty}\phi_t z=0$}
     \Bigr\}
\end{equation}
is a graph over the radius $R=\frac{\rho}{2}$ ball $B^+_R\subset E^+$ and this graph is
tangent to $E^+$ at $0$. More precisely, there is a $C^r$ map
\begin{equation}\label{eq:graph-G}
     \Gg^\infty
     =\left(G^\infty,id \right)
     :B^+_R\to E^-\oplus E^+
     ,\quad
     G^\infty(0)=0
     ,\quad
     dG^\infty(0)=0,
\end{equation}
whose image is a neighborhood of $0$ in $W^s(0,B_{\rho_0})$. Moreover, it holds that
$
     \Norm{\phi_t\xi_0}\le \rho e^{-t\lambda}
$
for every forward time $t\ge 0$ and uniformly in $\xi_0\in\Gg^\infty(B^+_R)$.
\end{theorem}

The proof of Theorem~\ref{thm:loc-st-mf-thm} is
by the {\bf contraction method}. In fact, our proof of the Backward $\lambda$-Lemma,
Theorem~\ref{thm:backward-lambda-Lemma}, presented below
generalizes the contraction method proof from invariant manifolds,
which is well known, to invariant foliations.
The details missing in the following sketch of
proof can be easily recovered by formally setting $T=\infty$
in the proof of Theorem~\ref{thm:backward-lambda-Lemma}.
\\
Pick $z_+\in E^+$ near $0$. Our object of interest
is a flow line $\xi:[0,\infty)\to B_{\rho_0}$
whose initial value projects to $z_+$ under $\pi_+$
and which converges to $0$, as $t\to\infty$.
For any sufficiently small constant
$\rho\in(0,\frac{\rho_0}{2})$ the complete metric space
\begin{equation}\label{eq:exp norm}
\begin{split}
     Z=Z_{\lambda,\rho}
     :=\Bigl\{\xi\in  
     C^0([0,\infty),T_xM)\,\,\big|\,\,
     \Norm{\xi}_{\exp}
     :=\sup_{t\ge 0} e^{t\lambda}
     \Norm{\xi(t)}
     \le\rho
     \Bigr\},
\end{split}
\end{equation}
carries a contraction defined by
\begin{equation*}
     (\Psi_{z_+}\xi)(t)
     =e^{-tA}z_+
     +\int_0^t e^{-(t-\sigma)A}\pi_+
     h(\xi(\sigma)) d\sigma
     -\int_t^\infty e^{-(t-\sigma)A^-}
     \pi_-h(\xi(\sigma)) d\sigma.
\end{equation*}
Moreover, by the representation formula~(\ref{eq:representation-formula})
for $T=\infty$ the (unique) fixed point $\xi_{z_+}$ is our object of
interest. For $R=\frac{\rho}{2}<\frac{\rho_0}{4}$ the map
\begin{equation*}
\begin{split}
     G^\infty: B^+_R
     &\to E^-,
      \qquad\qquad\quad
      B^+_R
      :=\left\{z_+\in E^+:
      \Norm{z_+}\le\rho/2
      \right\},
   \\
     z_+
     &\mapsto \pi_-\left( \xi_{z_+}(0)\right)
\end{split}
\end{equation*}
has the properties asserted by Theorem~\ref{thm:loc-st-mf-thm};
see also~\cite[Sec.~3.6]{chow:1982a}.

\begin{remark}
Further methods to prove the \emph{local} (un)stable manifold theorems:
\begin{itemize}
\item
  Graph transform method: A geometrically appealing
  method, nicely sketched in~\cite[p.80]{palis:1982a};
  for details see~\cite{shub:1987a}.
\item
  Irwin's space of sequences:
  See former two references. Another excellent
  reference for those who care about details
  is Zehnder's recent book~\cite{zehnder:2010a}.
\end{itemize}
\end{remark}

\begin{proof}[Proof of the stable manifold Theorem~\ref{thm:st-mf-thm}
(Henry~{\cite[Thm.~6.1.9]{henry:1981a}})]
\mbox{} \\ By Theorem~\ref{thm:loc-st-mf-thm} the stable manifold
is locally near $x$ a $C^r$ submanifold of $M$ of dimension $n-k$ where
$k$ is the Morse index of $x$ and $n=\dim M$.
Thus there is a neighborhood $\Sigma$ of $x$ in $W^s(x)$ which is
represented as a zero set $\{y\in U\mid h(y)=0\}$ where $U$ is an open set in $M$
and $h:U\to\R^k$ is a $C^r$ map such that $dh_y:T_yM\to\R^k$
is surjective at each point
$$
     y\in h^{-1}(0)=U\cap\Sigma=U\cap W^s(x).
$$
Note that here only the first identity is part of the submanifold
property of $\Sigma$. To obtain the second identity choose $U$
smaller, if necessary, and use the fact that a flow line of a
\emph{gradient} flow cannot come back to itself asymptotically.

Nearby arbitrary elements $q_0$ of $W^s(x)$ one obtains local submanifold charts as follows.
Pick $T\ge0$ such that $\varphi_Tq_0\in\Sigma$ and consider the $C^r$ map
$$
     h\circ\varphi_T:U_T\stackrel{\varphi_T}{\longrightarrow}
     U\stackrel{h}{\longrightarrow}\R^k,\qquad
     U_T:={\varphi_T}^{-1}(U),
$$
whose zero set is $U_T\cap W^s(x)$; see Figure~\ref{fig:fig-Henry-proof}.
The pre-image $U_T$ is an \emph{open} neighborhood of $q_0$ in $M$
since $U$ is open and $\varphi_T$ is continuous.
By the regular value theorem it remains to show that the map
$$
     d(h\circ\varphi_T)_q:T_qM\stackrel{d(\varphi_T)_q}{\longrightarrow}
     T_{y=\varphi_Tq}M\stackrel{dh_y}{\longrightarrow}\R^k
$$
is surjective whenever $q\in(h\circ\varphi_T)^{-1}(0)=U_T\cap W^s(x)$.
Since $dh_{\varphi_Tq}$ is surjective it suffices to show that
$d(\varphi_T)_q$ is surjective.\footnote{
  In infinite dimensions~\cite{henry:1981a} the operator $d\varphi_T|q$ is not
  surjective in general, but still admits dense image. This
  suffices to show surjectivity of the composition since
  $k$ is finite.
  }
The following argument avoids backward flows.
In finite dimensions surjectivity of $d(\varphi_T)_q$
is equivalent to dense range which itself is equivalent,
even in the general Banach space case,
to injectivity of the adjoint (or transposed) operator $(d\varphi_T|_q)^*$.
But the latter is equivalent to \emph{backward uniqueness}
of the (\emph{forward} Cauchy problem associated to the) \emph{adjoint equation}.
In our case $A=A^*$ so the adjoint equation is just the linearized
equation itself. But an ODE associated to a Lipschitz continuous
vector field exhibits forward and backward uniqueness; see e.g.~\cite{Agarwal:1993a}.
\end{proof}
\begin{figure}
  \centering
  \includegraphics
                             {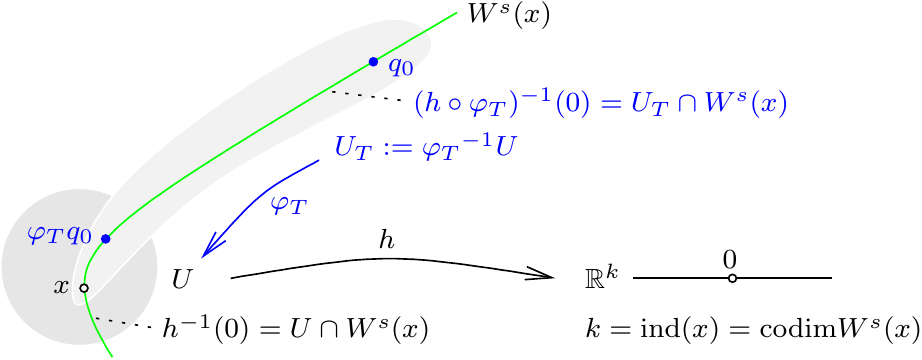}
  \caption{Henry's proof: pull \emph{back} coordinate charts by
    \emph{forward} time-$T$-map}
  \label{fig:fig-Henry-proof}
\end{figure}

\subsection*{Convenient coordinates}
Suppose the local setup of~(H1--H2) in Definition~\ref{def:loc-coord-lambda}.
Namely fix a constant $\lambda$ in the spectral gap $(0,d)$
and recall that the local flow $\phi_t$ acting on the ball $B_{\rho_0}$
in $T_xM=X=E^-\oplus E^+$ is generated by the ODE
$
     \dot\zeta+A\zeta=h(\zeta)
$
where the non-linearity $h$ is given by~(\ref{eq:f}). 
Consider the balls $B^-_R\subset E^-$ and $B^+_R\subset E^+$
of radius $R(x,\rho_0,\lambda)=\frac{\rho}{2}<\frac{\rho_0}{4}$
with $B^-_R\times B^+_R\subset B_{\rho_0/2}$
and the $C^r$ graph maps $F^\infty:B^-_R\to E^+$ and
$G^\infty:B^+_R\to E^-$ provided by the local (un)stable manifold
Theorems~\ref{thm:loc-unst-mf-thm} and~\ref{thm:loc-st-mf-thm}.
Use the fact that the graph $\Ff^\infty(B^-_R)$ is tangent to $E^-$ at $0$,
similarly for $\Gg^\infty(B^+_R)$, to see that choosing
the radius $R>0$ smaller, if necessary, one can arrange
that both graphs simultaneously satisfy inclusions
$$
     \Ff^\infty(B^-_R)\subset\left(B^-_R\times B^+_R\right),\qquad
     \Gg^\infty(B^+_R)\subset\left(B^-_R\times B^+_R\right). 
$$
On the other hand, by Corollary~\ref{cor:desc-disk} and Remark~\ref{rem:asc-disk}
there is a (small) constant $2\varsigma>0$ such that there are inclusions
of descending and ascending disks
\begin{equation}\label{eq:F-G-inc}
     W^u_{2\varsigma}\subset \Ff^\infty(B^-_R),\qquad 
     W^s_{2\varsigma}\subset \Gg^\infty(B^+_R).
\end{equation}
Set
\begin{equation}\label{eq:new-disks}
     D^-:=\pi_- W^u_{2\varsigma},\qquad D^+:=\pi_+ W^s_{2\varsigma},\qquad V:=D^-\times D^+.
\end{equation}
Then by the graph property of $\Ff^\infty$ and $\Gg^\infty$ it holds that
$     \Ff^\infty(D^-)=W^u_{2\varsigma}
$
and
$
     \Gg^\infty(D^+)=W^s_{2\varsigma}
$
as illustrated by the left part of Figure~\ref{fig:fig-flattening}.
Consequently $D^\pm$ is a {\bf disk}, that is a set diffeomorphic to a
closed ball. Note that
\begin{equation}\label{eq:un-stab-id}
\begin{aligned}
     W^u_{2\varsigma}&=W^u(0,V)=V\cap W^u(0,B_{\rho_0}),
     \\
     W^s_{2\varsigma}&=W^s(0,V)\hspace{.02cm}=V\cap W^s(0,B_{\rho_0}),
\end{aligned}
\end{equation}
by negative and positive invariance under $\phi$ of $W^u_{2\varsigma}$ and
$W^s_{2\varsigma}$, respectively.
Following Palis and de~Melo~\cite[Ch.~2 \S 7]{palis:1982a} observe that the $C^r$ map
$$
     \vartheta(x,y):=\left(x-G^\infty(y),y-F^\infty(x)\right)
$$
defined on $B^-_R\times B^+_R$ satisfies
$\vartheta(0)=0$ and $d\vartheta(0)=\1$. In particular, it is a
diffeomorphism locally near the fixed point $0$. 
Choosing $R$ and $2\varsigma$ smaller,  if necessary, we assume without
loss of generality that $\vartheta$ is a diffeomorphism onto its image.
Note that $\vartheta(x,0)=(x,-F^\infty(x))$ and $\vartheta(0,y)=(-G^\infty(y),y)$. Consequently
the map $\vartheta$ diffeomorphically maps $\Ff^\infty(D^-)=W^u_{2\varsigma}$ to the
disk $D^-$ and similarly for $W^s_{2\varsigma}$ and $D^+$; see Figure~\ref{fig:fig-flattening}.
\begin{figure}
  \centering
  \includegraphics
                             {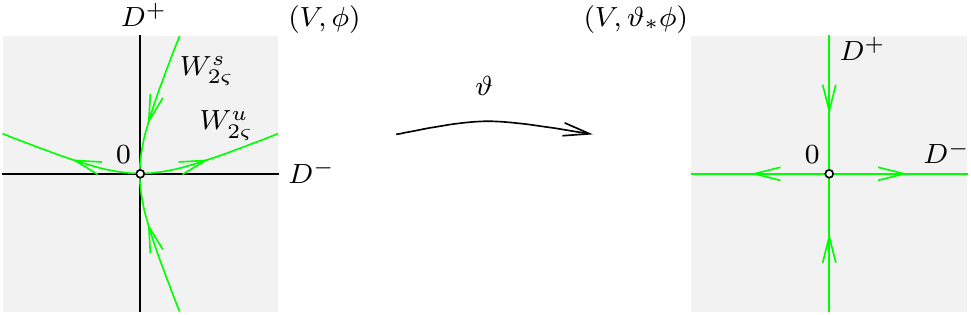}
  \caption{Local (un)stable manifolds get flattened out}
  \label{fig:fig-flattening}
\end{figure}
Our new local model will be the product of disks $V:=D^-\times D^+$
acted upon by the new local flow
$\vartheta_*\phi_t=\phi_t\circ\vartheta\circ{\phi_t}^{-1}$
(denoted again by $\phi_t$ in Remark~\ref{rem:flat-inv-mfs})
whose local unstable and stable manifolds are $D^-$ and $D^+$, respectively.
Note that the latter represent the descending disk $W^u_{2\varsigma}(x)$
and the ascending disk $W^s_{2\varsigma}(x)$ in this new local model.
The ODE which generates the flow $\vartheta_*\phi_t$
arises by re-doing Lemma~\ref{le:conv-coord}
starting from the ``Ansatz'' $u(t)=\exp_x\left({\vartheta}^{-1}\xi(t)\right)$ for $\xi$.
Obviously the new ODE is similar to~(\ref{eq:f}) involving, in addition,
the diffeomorphism $\vartheta$ and its linearization.
Since $\vartheta$ is of class $C^r$ and defined on the compact set $V$
it admits uniform $C^r$-bounds. Therefore relevant estimates for the new and
the old ODE are equivalent up to constants.
This justifies to simplify notation as follows.

\begin{remark}[Coordinates with flattened local manifolds]\label{rem:flat-inv-mfs}
By the discussion above and
in order to simplify notation we will, without loss of generality,
continue to work with the ODE~(\ref{eq:f}), that is we omit
$\vartheta$ in our formulas, and assume that the local (un)stable
manifolds $W^u_{2\varsigma}$ and $W^s_{2\varsigma}$ for the (new) local flow $\phi$ on $V$
are disks $D^-\subset E^-$ and $D^+\subset E^+$. To summarize we
assume that
\begin{equation}\label{eq:DDD}
     V:=D^-\times D^+\subset B^-_R\times B^+_R\subset B_1,\quad
     D^-=W^u_{2\varsigma},\quad D^+=W^s_{2\varsigma},
\end{equation}
as illustrated by Figure~\ref{fig:fig-local-setup-NEW} and where
$\varsigma>0$ has been fixed according to~(\ref{eq:F-G-inc}).
\end{remark}

\section{Proof of backward $\lambda$-Lemma} 
Suppose $\IND(x)=k\in\{1,\dots,n-1\}$, so the smallest
and largest eigenvalue of the Hessian operator $A=A_x$
satisfy $\lambda_1<0<\lambda_n$.
Consider the continuous function $\kappa(\rho)$ with $\kappa(0)=0$ and the
Lipschitz constant $\kappa_*>0$ of the non-linearity $h$ provided
by Lemma~\ref{le:f}.
Pick an exponential decay rate $\lambda\in(0,d)$ for the
elements of the complete metric space $Z^T$ to be defined below
and consider the two constants
$
     \delta\in\left(0,\min\{1,\tfrac{d-\lambda}{2}\}\right)
$
and
$
     \mu\in\left(\lambda,\tfrac{d+\lambda}{2}\right)
     \subset(\lambda,d)
$
defined in Figure~\ref{fig:fig-spectral-gap}.
To ensure the second of the two endpoint conditions in~(\ref{eq:xi-end}) set
\begin{equation}\label{eq:T_1}
     T_1=T_1(\lambda,\varkappa)
     :=-\frac{\ln\varkappa}{\lambda}\ge 0.
\end{equation}
Fix $\rho=\rho(\lambda)\in(0,1)$ sufficiently small such that
\begin{equation}\label{eq:rho-backward}
     \kappa(\rho)\left(
     \frac{4}{\lambda}+\frac{1}{\delta}+1\right)
     \le\frac{1}{8}
\end{equation}
and such that the closed $\rho$-neighborhood $U_\rho(W^u_\varsigma)$
in $X$ of the descending disk is contained in $V=D^-\times D^+\subset B_1$; see~(\ref{eq:DDD}).
Observe that in this case the radius $\frac{\rho}{2}$ ball $B^+\subset E^+$ is contained in $D^+$.
Fix $T_2=T_2(\mu)>0$ such that $e^{-T_2 \mu/4}\le 1/8$; this will be used in Step~5. Set
\begin{equation}\label{eq:T_0}
     T_0=T_0(\lambda,\mu(\lambda),\varkappa)
     :=\max\{T_1,T_2,1\}
     \ge 1.
\end{equation}

\begin{remark}[Mixed Cauchy problem]
\label{rem:mixed-Cauchy}
The key observation to represent the pre-image 
${\phi_T}^{-1}\Dd_{z_-}$ under the time-$T$-map $\phi_T$ as a graph over the stable tangent space
$E^+$ is the fact that there is a well posed {\bf mixed Cauchy problem}. Namely, instead
of prescribing precisely the endpoint, by the representation
formula~(\ref{eq:representation-formula}) it makes sense to prescribe only the ($\pi_+$)-part
of the initial point, but in addition the  ($\pi_-$)-part of the endpoint.
This way one circumvents using a general backward flow -- only the
(by Convention~\ref{con:back-flow} legal) algebraic one on the unstable manifold is
used; cf.~\cite[Rmk.~3]{weber:2014a}.
\end{remark}

We prove below that for each $z_+\in E^+$ with $\norm{z_+} \le\rho/2$ there is precisely one
flow line $\xi=\xi_{z_-,z_+}^T$ with initial condition $\pi_+ \xi(0)=z_+$ and endpoint condition
$\xi(T)\in\Dd_{z_-}$. The latter is equivalent to
\begin{equation}\label{eq:xi-end}
     \pi_-\xi(T)={z_-}
     \quad\wedge\quad
     \Norm{\xi(T)-{z_-}}\le\varkappa.
\end{equation} 
The key step to determine the unique semi-flow line $\xi$ associated
to the triple $(T,{z_-},z_+)$ is to set up a strict contraction $\Psi^T$ on a complete
metric space $Z^T$ whose (unique) fixed point is $\xi$. More precisely, define
\begin{equation}\label{eq:ZT}
     Z^T
     =Z^{T,{z_-}}_{\lambda,\rho}
     :=\left\{
     \xi\in
     C^0([0,T],X)\,\colon
     \Norm{\xi-\phi_\cdot{z_-^T}}
     _{\exp}
     \le\rho\right\},
\end{equation}
where ${z_-^T}:=\phi_{-T}({z_-})$ and
\begin{equation}\label{eq:exp-T}
     \Norm{\xi}_{\exp}
     =\Norm{\xi}_{\exp,T}
     :=\max_{t\in[0,T]} 
     e^{t\lambda}
     \Norm{\xi(t)},
\end{equation}
and consider the operator
$\Psi^T=\Psi^T_{{z_-},z_+}$ defined for
$\xi\in Z^T$ by
\begin{equation}\label{eq:Psi}
\begin{split}
     \left(\Psi_{{z_-},z_+}^T
     \xi\right)(t)
    &:=e^{-tA}z_+
     +\int_0^t e^{-(t-\sigma)A}\pi_+
     h(\xi(\sigma))
     \, d\sigma
     \\
    &\quad
     +e^{-(t-T)A^-}{z_-}
     -\int_t^T e^{-(t-\sigma)A^-}\pi_-
     h(\xi(\sigma))
     \, d\sigma
\end{split}
\end{equation}
for every $t\in[0,T]$. The fixed points of $\Psi^T$ correspond to the desired flow
trajectories by Proposition~\ref{prop:representation-formula}.
In Step~1 and Step~2 below we show that $\Psi^T$ is a strict contraction on $Z^T$. Hence by the
Banach-Cacciopoli Fixed Point Theorem, see e.g.~\cite[\S 2.2]{chow:1982a},
it admits the unique fixed point $\xi_{{z_-},z_+}^T$ and we define the map
\begin{equation}\label{eq:G^T}
     G^T:S^u_\eps\times B^+\to E^-
     ,\qquad
     \left({z_-},z_+\right)\mapsto 
     \pi_-\xi_{{z_-},z_+}^T(0)
     =:G^T_{z_-}(z_+).
\end{equation}
For latter use we calculate for $t\in[0,T]$ and $\alpha>0$ the integrals
\begin{equation}\label{eq:int-1}
     \int_0^t e^{-(t-\sigma)\alpha}\, d\sigma
     =\frac{1-e^{-t\alpha}}{\alpha}
     ,\qquad
     \int_t^T e^{(t-\sigma)\alpha}\, d\sigma
     =\frac{1-e^{(t-T)\alpha}}{\alpha}.
\end{equation}
The proof takes six steps. Fix ${z_-}\in S^u_\eps$ and $z_+\in B^+$. Abbreviate $\Psi^T=\Psi_{{z_-},z_+}^T$.

\vspace{.2cm}
\noindent
{\bf Step 1.}
{\it For $T>0$ the set $Z^T$ equipped with the exp~norm~metric is a complete
metric space, any  $\xi\in Z^T$ takes values in $V\subset B_1$, and $\Psi^T$ acts on~$Z^T$.
}

\begin{proof}
For the compact domain
$[0,T]$ the space $C^0([0,T],X)$ is complete with respect to the
sup norm, hence with respect to the exp norm as both norms
are equivalent by compactness of $[0,T]$. By its definition the subset $Z^T$
of $C^0([0,T],X)$ is closed with respect to the exp norm. 
Since $U_\rho(W^u_\eps)\subset U_\rho(W^u_\varsigma)\subset V$
by our choice of $\rho$, the elements of $Z^T$ take values in $V\subset B_1$.

To see that $\Psi^T$ acts on $Z^T$
we need to verify that $\Psi^T\xi$ is continuous
and satisfies the exponential decay condition
whenever $\xi\in Z^T$.
By definition $\Psi^T\xi$ is a sum of four terms
each of which is continuous as a map $[0,T]\to X$.
For terms one and three continuity, in fact
smoothness, follows from definition~(\ref{eq:e^tA})
of the exponential as a power series. Continuity
of both integral terms, terms two and four, uses
the same argument. Denote term two by $F(t)$
from now on. Continuity of $F:[0,T]\to X$ and the
fact that $F(0)=0$ (used in Step~3 below) both
follow from continuity and boundedness of the map
$\tilde{h}:=\pi_+\circ h\circ\xi:[0,T]\to E^+$
which holds true since $\xi:[0,T]\to X$ is
continuous and bounded by definition of $Z^T$
and so is the non-linearity $h$ by the Lipschitz
Lemma~\ref{le:f}.

We prove exponential decay. For $t\in[0,T]$ consider the flow trajectory given by
$\phi_t z_-^T$ where ${z_-^T}:=\phi_{-T}({z_-})$.
By the representation formula~(\ref{eq:representation-formula}) it satisfies
\begin{equation}\label{eq:phi-s}
\begin{split}
     \phi_t{z_-^T}
    &=\int_0^t e^{-(t-\sigma)A}\pi_+
     h(\phi_\sigma {z_-^T})\, d\sigma
   \\
    &\quad
     +e^{-(t-T)A^-}{z_-}
     -\int_t^T e^{-(t-\sigma)A^-}\pi_-
     h(\phi_\sigma {z_-^T})\, d\sigma.
\end{split}
\end{equation}
Here we used that ${z_-}\in S^u_\eps=\p W^u_\eps$, thus
${z_-^T}=\phi_{-T}{z_-}$, lies in the (backward flow invariant) descending
disk $W^u_\eps\subset D^-$. Hence $\pi_+{z_-^T}=0$ and $\pi_-\phi_T {z_-^T}=\pi_-{z_-}={z_-}$.
By~(\ref{eq:Psi}) and~(\ref{eq:phi-s}) we get for $t\in[0,T]$ the estimate
\begin{equation}\label{eq:Psixi-bounded}
\begin{split}
    &e^{t\lambda}\Norm{\left(\Psi^T\xi\right)(t)
     -\phi_t{z_-^T}
     }
   \\
    &\le
     e^{t\lambda}\Norm{e^{-tA}z_+}
     +e^{t\lambda}\int_0^t e^{-(t-\sigma)\mu}
     \Norm{h(\xi(\sigma))-h(\phi_\sigma{z_-^T})
     } d\sigma
   \\
    &\quad
     +e^{t\lambda}\int_t^T e^{(t-\sigma)\mu}
     \Norm{h(\xi(\sigma))-h(\phi_\sigma{z_-^T})
     } d\sigma
   \\
\end{split}
\end{equation}
\begin{equation*}
\begin{split}
    &\le e^{t\lambda} e^{-t\mu}\Norm{z_+}
     +\kappa(\rho)
     \Norm{\xi-\phi_\cdot{z_-^T}}_{\exp}
     \int_0^t
     e^{-(t-\sigma)\delta}
     \, d\sigma
   \\
    &\quad
     +\kappa(\rho)
     \Norm{\xi-\phi_\cdot{z_-^T}}_{\exp}
     \int_t^T e^{(t-\sigma)(\lambda+\mu)}
     \, d\sigma
   \\
    &\le
     \frac{\rho}{2} e^{-t\delta}
     +\kappa(\rho)
     \left(\frac{1}{\delta}+\frac{1}{\lambda+\mu}
     \right)\rho
     \le\rho
\end{split}
\end{equation*}
where we used the exponential decay estimates in
Proposition~\ref{prop:exp-est}. To the
non-linearity $h$ we applied the Lipschitz
Lemma~\ref{le:f} to bring in the
constant~$\kappa(\rho)$. We also used the fact that
the norm of a projection is bounded by $1$.
Moreover, we multiplied both
integrands by $e^{-\sigma\lambda} e^{\sigma\lambda}$
to obtain the exp norms which are bounded by
$\rho$ by definition of $Z^T$. Inequality three
uses that $z_+\in B^+$ and the integral
estimates in~(\ref{eq:int-1}). The final step is by
smallness~(\ref{eq:rho-backward}) of $\rho$.
\end{proof}

\noindent
{\bf Step 2.}
{\it For $T>0$ the map $\Psi^T$ is a contraction on $Z^T$. Each image point $\Psi^T\xi$ satisfies
the initial condition $\pi_+\left(\Psi^T\xi\right)(0) =z_+$ and for $T\ge T_1$ also the endpoint 
conditions~(\ref{eq:xi-end}), i.e. it hits $\Dd_{z_-}=\{{z_-}\}\times B^+_\varkappa$ at time $T$.
}

\begin{proof}
Pick $\xi_1,\xi_2\in Z^T$. Then similarly to~(\ref{eq:Psixi-bounded}) we obtain that
\begin{equation*}
\begin{split}
     e^{t\lambda}\Norm{\left(\Psi^T\xi_1\right)(t)
     -\left(\Psi^T\xi_2\right)(t)}
     &\le e^{t\lambda}
      \int_0^t e^{-(t-\sigma)\mu}
      \Norm{h(\xi_1(\sigma))
      -h(\xi_2(\sigma))} d\sigma
   \\
     &\quad
      +e^{t\lambda}\int_t^T e^{(t-\sigma)\mu}
      \Norm{h(\xi_1(\sigma))
      -h(\xi_2(\sigma))} d\sigma
    \\
      &\le \kappa(\rho)
       \left(\frac{1}{\delta}
       +\frac{1}{\lambda+\mu} 
       \right)
       \Norm{\xi_1-\xi_2}_{\exp}
\end{split}
\end{equation*}
for every $t\in[0,T]$. Use smallness~(\ref{eq:rho-backward}) of $\rho$
and take the maximum over $t\in[0,T]$ to get that $\norm{\Psi^T\xi_1-\Psi^T\xi_2}_{\exp}
\le\frac{1}{2}\norm{\xi_1-\xi_2}_{\exp}$.

To obtain the identities $\pi_+\left(\Psi^T\xi\right)(0)=z_+$ and
$\pi_-\left(\Psi^T\xi\right)(T)={z_-}$ just
set $t=0$ in the definition~(\ref{eq:Psi}) of $\Psi^T$
and use the identities $\pi_+\pi_-=\pi_-\pi_+=0$,
continuity of the exponential series~(\ref{eq:e^tA}),
and continuity and boundedness of both integrands.
Concerning the second endpoint condition
in~(\ref{eq:xi-end}) assume $T\ge T_1$ and evaluate
estimate~(\ref{eq:Psixi-bounded}) at $t=T$ to get that
$$
     \Norm{\left(\Psi\xi\right)(T)
     -{z_-}}_X
     \le \rho
     e^{-T\lambda}
     \le e^{-T_1\lambda}
     =\varkappa
$$
where the last step is by $\rho\le\rho_0/2\le 1$ and definition~(\ref{eq:T_1}) of $T_1$.
\end{proof}

\noindent
{\bf Step 3.}
{\it For $T>0$ the map $G^T:S^u_\eps\times B^+\to E^-$ defined by~(\ref{eq:G^T}) is of class $C^r$ and,
for each ${z_-}\in S^u_\eps$, the map $G^T_{z_-} :=G^T({z_-},\cdot):B^+\to E^-$ satisfies
$$
     G^T_{z_-}(0)
     =\phi_{-T}({z_-})
     =:{z_-^T}
     ,\qquad
     \mathrm{graph}\, G^T_{z_-}
     =\left\{\xi_{{z_-},z_+}^T(0)
     \,\big|\, 
     z_+\in B^+\right\}.
$$
} 

\begin{proof}
By Step~2 and its proof the map
$$
     \Psi^T:
     S^u_\eps\times B^+\times Z^T\to Z^T,\qquad
     ({z_-},z_+,\xi)\mapsto\Psi^T_{{z_-},z_+}\xi
$$
is a uniform contraction on $Z^T$ with contraction factor $\tfrac12$. (Strictly speaking $Z^T$ depends on
${z_-}$, but the complete metric spaces associated to different ${z_-}$'s are naturally
isomorphic.) Observe that $\Psi^T$ is linear, hence smooth, in ${z_-}$ and in
$z_+$ and of class $C^r$ in $\xi$, because $h$ is of class $C^r$ by the Lipschitz Lemma~\ref{le:f}.
Hence by the uniform contraction principle, see e.g.~\cite[\S 2.2]{chow:1982a}, the map 
\begin{equation}\label{eq:unif-contraction-theta}
\begin{split}
     \theta=\theta^T:S^u_\eps\times B^+
   &\to Z^T
   \\
     ({z_-},z_+)
   &\mapsto
     \xi^T_{{z_-},z_+}
\end{split}
\end{equation}
which assigns to $({z_-},z_+)$  the unique fixed point of $\Psi^T_{{z_-},z_+}$ is of class
$C^r$ and so is its composition with the (linear) evaluation map $ev_0:Z^T\to X$, $\xi\mapsto\xi(0)$,
and the (linear) projection $\pi_-:X\to E^-$. But the composition of these maps is $G^T$ by
its definition~(\ref{eq:G^T}). So $G^T$, thus $\Gg$, is of class $C^r$ in ${z_-}$ and $z_+$.

Consider the flow trajectory $\eta:[0,T]\to X$, $t\mapsto\phi_t{z_-^T}$,
which runs from ${z_-^T}$ to ${z_-}$ inside the (backward invariant) descending
disk $W^u_\eps\subset D^-$. Hence $\pi_+\eta(0)=0$ and $\pi_-\eta(T)=\eta$.
Thus $\eta=\xi^T_{{z_-},0}$ by uniqueness of the fixed point. Hence
$
     G^T_{z_-}(0)
     :=\pi_-\xi^T_{{z_-},0}(0)
     =\pi_-\eta(0)
     =\pi_-{z_-^T}
     ={z_-^T}
$.
To get the desired representation of $\mathrm{graph}\, G^T_{z_-}$ observe that
\begin{equation}\label{eq:G=xi}
     \Gg^T_{z_-}(z_+)
     :=\left( G^T_{z_-}(z_+),z_+\right)
     =\left(\pi_-\xi^T_{{z_-},z_+}(0),
     \pi_+ \xi^T_{{z_-},z_+}(0)\right)
     =\xi^T_{{z_-},z_+}(0)
\end{equation}
by Definition~(\ref{eq:G^T}). The first identity also uses the fixed point
property and the initial condition from Step~2. The final identity is by $\pi_-\oplus\pi_+=\1_X$.
\end{proof}

\noindent
{\bf Step 4.}
{\it The map $T\mapsto\Gg(T,{z_-},z_+)$ is Lipschitz continuous.
If $f$ is of class $C^{2,1}$ near $x$, then the derivative
$T\mapsto\frac{d}{dT}\Gg(T,{z_-},z_+)$ is also Lipschitz continuous.
} 

\begin{proof}
{\it We prove that $\Gg$ is Lipschitz continuous in $T$.}
Fix $T\ge T_0>0$, ${z_-}\in S^u_\eps$, and $z_+\in{B^+}$. Consider the fixed point
$\xi^T=\xi^T_{{z_-},z_+}$ of the strict contraction $\Psi^T=\Psi^T_{z_-,z_+}$ defined
by~(\ref{eq:Psi}). The fixed point of $\Psi^{T+\tau}$ is given by
\begin{equation}\label{eq:def-xi}
\begin{split}
     \xi^{T+\tau}(t)
    &=e^{-tA}z_+
     +\int_0^t e^{-(t-\sigma)A}\pi_+ h(\xi^{T+\tau}(\sigma)) \, d\sigma
     \\
    &\quad
     +e^{-(t-T-\tau)A^-}{z_-}
     -\int_t^{T+\tau} e^{-(t-\sigma)A^-}\pi_- h(\xi^{T+\tau}(\sigma)) \, d\sigma.
\end{split}
\end{equation}
For $t\in[0,T]$ and $\tau\ge0$ we obtain, similarly to~(\ref{eq:Psixi-bounded}), the estimate
\begin{equation*}
\begin{split}
    &\Norm{\xi^{T+\tau}(t)-\xi^T(t)}
   \\
    &\le\int_0^t e^{-(t-\sigma)\mu}
     \Norm{h(\xi^{T+\tau}(\sigma))-h(\xi^T(\sigma))} \, d\sigma
     +\bigl\| \bigl(e^{\tau A^-}-\1\bigr) e^{-(t-T)A^-}{z_-}\bigr\|
   \\
    &
     +\int_t^T e^{(t-\sigma)\mu}
     \Norm{h(\xi^{T+\tau}(\sigma))-h(\xi^T(\sigma))} d\sigma
     +\int_T^{T+\tau} e^{(t-\sigma)\mu} \Norm{h(\xi^{T+\tau}(\sigma))} d\sigma
\end{split}
\end{equation*}
\begin{equation*}
\begin{split}
    &\le \kappa(\rho)
     \Norm{\xi^{T+\tau}-\xi^T}_{C^0([0,T],X)}
     \left(\int_0^t e^{-(t-\sigma)\mu} \, d\sigma+\int_t^T e^{(t-\sigma)\mu}\, d\sigma\right)
   \\
    &\quad
     +\tau\abs{\lambda_1}\cdot \underbrace{e^{(t-T)\mu}}_{\le 1} \Norm{{z_-}}
     +\kappa(\rho)\int_T^{T+\tau} e^{(t-\sigma)\mu}\, d\sigma
   \\
    &\le \kappa(\rho)\frac{2}{\mu}\Norm{\xi^{T+\tau}-\xi^T}_{C^0([0,T],X)}
     +\tau\abs{\lambda_1}
     + \kappa(\rho) \frac{1-e^{-\tau\mu}}{\mu}
   \\
    &\le\frac{1}{8}
     \Norm{\xi^{T+\tau}-\xi^T}_{C^0([0,T],X)}
     +\tau\left(\abs{\lambda_1}+1\right).
\end{split}
\end{equation*}
Inequality two uses the Lipschitz Lemma~\ref{le:f} for $f$ and the exponential
estimates in Proposition~\ref{prop:exp-est}.
Moreover, we used the fact that $e^{-s\mu}\le 1$ to obtain
\begin{equation}\label{eq:1-e}
     \frac{1-e^{-\mu\tau}}{\mu}
     =\int_0^\tau e^{-s\mu}\, ds
     \le\tau.
\end{equation}
The identity
\begin{equation}\label{eq:ident-1}
     e^{\tau A^-}-\1=A^-\int_0^\tau e^{\sigma A^-}\, d\sigma
\end{equation}
follows by definition~(\ref{eq:e^tA}) of the exponential as a series. Together with
 the exponential estimates in Proposition~\ref{prop:exp-est} we get that
\begin{equation}\label{eq:e-1}
     \Norm{e^{\tau A^-}-\1}
     =\Norm{A^-\int_0^\tau e^{\sigma A^-}\, d\sigma}
     \le\Abs{\lambda_1}\int_0^\tau e^{-\sigma\mu}\, d\sigma
     \le\Abs{\lambda_1}\tau.
\end{equation}
Coming back to inequality two, it remains to explain the estimate for
term four. Here we used that $\xi^{T+\tau}\in Z^{T+\tau}$
takes values in $B_1$ by Step~1.
\\
Inequality three uses~(\ref{eq:int-1}) for the integrals and the fact that $\norm{{z_-}}\le 1$.
Inequality four uses estimate~(\ref{eq:1-e}) and the smallness assumption~(\ref{eq:rho-backward})
on $\rho$ by which $\kappa(\rho)\le1$. Now take the
supremum over $t\in[0,T]$ to obtain that
\begin{equation}\label{eq:xi-xi}
     \Norm{\xi^{T+\tau}-\xi^T}_{C^0([0,T],X)}
     \le {c_1}\tau
\end{equation}
with constant ${c_1}=2(\abs{\lambda_1}+1)$. By~(\ref{eq:G=xi}) this shows that
\begin{equation}\label{eq:C0-conv}
     \Norm{\Gg^{T+\tau}_{z_-}(z_+)-\Gg^T_{z_-}(z_+)}
     =\Norm{\xi^{T+\tau}(0)-\xi^T(0)}
     \le {c_1}\tau
\end{equation}
for all $T\ge T_0$ and $\tau\ge 0$. In other words, the map
$T\mapsto \Gg(T,{z_-},z_+)=\Gg^T_{z_-}(z_+)$
is Lipschitz. The difference $\xi^{T+\tau}-\xi^T$
is illustrated by~\cite[Fig.~5]{weber:2014a}.

{\it We prove that the map $T\mapsto\frac{d}{dT}\Gg(T,{z_-},z_+)$ is Lipschitz continuous.}
Set
\begin{equation}\label{eq:Theta-T}
     \Theta^T(t)=\Theta^T_{{z_-},z_+}(t)
     :=\left.\tfrac{d}{d\tau}\right|_{\tau=0}
     \xi^{T+\tau}_{{z_-},z_+}(t)
\end{equation}
for every $t\in[0,T]$. Calculation shows that this derivative is given by
\begin{equation*}
\begin{split}
     \Theta^T(t)
    &=\int_0^t e^{-(t-\sigma)A}\pi_+
     \left(dh|_{\xi^T(\sigma)}\circ \Theta^T(\sigma)\right)
     d\sigma
     +A^- e^{-(t-T)A^-}{z_-}
     \\
    &\quad
     -e^{-(t-T)A^-}\pi_- h(\xi^T(T))
     -\int_t^T e^{-(t-\sigma)A^-}\pi_-
     \left(dh|_{\xi^T(\sigma)}\circ \Theta^T(\sigma)\right)
     d\sigma.
\end{split}
\end{equation*}
Since $\frac{d}{dT}\Gg(T,{z_-},z_+)=\Theta^T(0)$
by~(\ref{eq:G=xi}), it remains to show that the map
$T\mapsto \Theta^T(0)\in X$ is Lipschitz continuous.
By definition of $\Theta^T$ we get the identity
\begin{equation*}
\begin{split}
     \Theta^{T+\tau}(t)-\Theta^T(t)
    &=\int_0^t e^{-(t-\sigma)A}\pi_+
     dh|_{\xi^{T+\tau}(\sigma)}\circ
     \left( \Theta^{T+\tau}(\sigma)-\Theta^T(\sigma)\right)
     \, d\sigma
  \\
    &\quad
     +\int_0^t e^{-(t-\sigma)A}\pi_+
     \left(
     dh|_{\xi^{T+\tau}(\sigma)}-dh|_{\xi^T}(\sigma)
     \right)\circ 
     \Theta^T(\sigma)
     \, d\sigma
   \\
    &\quad
     +\bigl(e^{\tau A^-}-\1\bigr)
     A^- e^{-(t-T)A^-}{z_-}
   \\
    &\quad
     -\bigl(e^{\tau A^-}-\1\bigr)
     e^{-(t-T)A^-}\pi_- h(\xi^{T+\tau}(T+\tau))
   \\
    &\quad
     -
\underline{
     e^{-(t-T)A^-}\pi_-
     \left(
     h(\xi^{T+\tau}(T+\tau))-h(\xi^T(T))
     \right)
}
   \\
    &\quad
     -\int_t^T e^{-(t-\sigma)A^-}\pi_-
     dh|_{\xi^{T+\tau}(\sigma)}\circ
     \left( \Theta^{T+\tau}(\sigma)-\Theta^T(\sigma)\right)
     \, d\sigma
  \\
    &\quad
     -\int_t^T e^{-(t-\sigma)A^-}\pi_-
     \left(
     dh|_{\xi^{T+\tau}(\sigma)}-dh|_{\xi^T}(\sigma)
     \right)\circ 
     \Theta^T(\sigma)
     \, d\sigma
   \\
    &\quad
     -\int_T^{T+\tau}
     e^{-(t-\sigma)A^-}\pi_-
     dh|_{\xi^{T+\tau}(\sigma)}\circ \Theta^{T+\tau}(\sigma)
     \, d\sigma
\end{split}
\end{equation*}
for all $t\in[0,T]$ and $\tau\ge0$. To obtain lines one and two add zero; similarly
for lines four and five and lines six and seven. Abbreviate the $C^0([0,T],X)$ norm
by $\norm{\cdot}_{C^0_T}$. Combine line one with line six and line two
with line seven to obtain, similarly as above and again abbreviating
${c_1}=2(\abs{\lambda_1}+1)$, the estimate
\begin{equation*}
\begin{split}
    &\Norm{\Theta^{T+\tau}(t)-\Theta^T(t)}_X
   \\
    &\le \kappa(\rho)\Norm{\Theta^{T+\tau}-\Theta^T}_{C^0_T}
     \left(\int_0^t e^{-(t-\sigma)\mu}\, d\sigma+\int_t^T e^{(t-\sigma)\mu}\, d\sigma\right)
   \\
    &\quad
     +\kappa_*\Norm{\xi^{T+\tau}-\xi^T}_{C^0_T}\Norm{\Theta^T}_{C^0_T}
     \left(\int_0^t e^{-(t-\sigma)\mu}\, d\sigma+\int_t^T e^{(t-\sigma)\mu}\, d\sigma\right)
   \\
    &\quad
     +\tau\Abs{\lambda_1}^2\Norm{{z_-}}
     +\tau\Abs{\lambda_1}\kappa(\rho)\Norm{\xi^{T+\tau}(T+\tau)}
     +
\underline{
     \kappa(\rho)\Norm{\xi^{T+\tau}(T+\tau)-\xi^T(T)}
}
   \\
    &\quad
     +\kappa(\rho)\Norm{\Theta^{T+\tau}}_{C^0_{T+\tau}}
     \int_T^{T+\tau} e^{(t-\sigma)\mu}\, d\sigma
   \\
    &\le \kappa(\rho)\Norm{\Theta^{T+\tau}-\Theta^T}_{C^0_T}\frac{2}{\mu}
     +\kappa_*({c_1}\tau){c_1} \frac{2}{\mu}
     +\tau\Abs{\lambda_1}\left(\Abs{\lambda_1}+1\right)
   \\
    &\quad
     +
\underline{
     \kappa(\rho)\tau\left(c_1+2\lambda_n+1\right)
}
     +\kappa(\rho) {c_1}\frac{1-e^{-\tau\mu}}{\mu}
   \\
\end{split}
\end{equation*}
\begin{equation*}
\begin{split}
    &\le\frac{1}{8}
     \Norm{\Theta^{T+\tau}-\Theta^T}_{C^0_T}
     +\tau\left(\frac{2\kappa_*{c_1}^2}{\mu}+\frac{{c_1}^2}{4}
     +
\underline{
     \left(c_1+2\lambda_n+1\right)
}
     +{c_1}\right)
\end{split}
\end{equation*}
for all $t\in[0,T]$ and $\tau\ge 0$. Inequality one uses the
exponential estimates in Proposition~\ref{prop:exp-est} and the Lipschitz
Lemma~\ref{le:f} for $dh$. To obtain line three we used~(\ref{eq:e-1})
and we added zero in the form of $h(0)$, thus bringing in $\kappa(\rho)$.
\\
Inequality two uses the following arguments. Estimate the first pair of integrals by $\frac{2}{\mu}$
using~(\ref{eq:int-1}); similarly for the second pair. Apply estimate~(\ref{eq:xi-xi}) and
recall that ${z_-}\in B_1$ by our local setup. Use again~(\ref{eq:xi-xi}) to conclude that
\begin{equation}\label{eq:Theta-good!}
     \Norm{\Theta^T(t)}
     =\Norm{\tfrac{d}{dT}\xi^{T}(t)}
     =\lim_{\tau\to 0}
     \frac{\Norm{\xi^{T+\tau}(t)-\xi^T(t)}}{\tau}
     \le{c_1}
\end{equation}
whenever $t\in[0,T]$. Thus
$
     \norm{\Theta^T}_{C^0_T}\le {c_1}
$
and, of course, the same is true when $T$ is replaced by $T+\tau$. By Step~1
the elements of the complete metric spaces $Z^T$ (and $Z^{T+\tau}$) take
values in $B_1$. We also used that $\kappa(\rho)\le1$
by~(\ref{eq:rho-backward}) and that $e^{(t-T)\mu}\le1$.
{\it The estimate for the difference
$\xi^{T+\tau}(T+\tau)-\xi^T(T)\in X$ in line four
will be carried out separately below};
see~(\ref{eq:xi-xi-T-tau}) for the result.
\\
Inequality three uses the smallness assumptions~(\ref{eq:rho-backward}) on $\rho$ and
estimate~(\ref{eq:1-e}). 

Now take the supremum over $t\in[0,T]$ to get the estimate
\begin{equation}\label{eq:V-V}
     \Norm{\Theta^{T+\tau}-\Theta^T}_{C^0([0,T],X)}
     \le C \tau
\end{equation}
for all $T\ge T_0$ and $\tau\ge0$ and where
$C=C(\lambda_1,\lambda_n,\kappa_*,\mu^{-1})$ is a constant. So
\begin{equation*}
     \Norm{\tfrac{d}{dT}\Gg(T+\tau,{z_-},z_+)-\tfrac{d}{dT}\Gg(T,{z_-},z_+)}
     =\Norm{\Theta^{T+\tau}(0)-\Theta^T(0)}
     \le C\tau
\end{equation*}
which means that the map $T\mapsto \frac{d}{dT}\Gg(T,{z_-},z_+)$ is Lipschitz continuous.

{\it As mentioned above it remains to estimate the norm of the difference:}
\begin{equation*}
\begin{split}
     \xi^{T+\tau}(T+\tau)-\xi^T(T)
    &=\left(e^{-\tau A}-\1\right) e^{-T A}z_+
  \\
    &\quad
     +\int_0^T e^{-(T+\tau-\sigma)A}\pi_+
     \left(
     h(\xi^{T+\tau}(\sigma))-h(\xi^T(\sigma))
     \right)
     \, d\sigma
  \\
    &\quad
     +\int_0^T
     \left(e^{-\tau A}-\1\right)
     e^{-(T-\sigma)A}\pi_+ h(\xi^T(\sigma))
     \, d\sigma
   \\
    &\quad
     +\int_T^{T+\tau}
     e^{-(T+\tau-\sigma)A}\pi_+
     h(\xi^{T+\tau}(\sigma))
     \, d\sigma.
\end{split}
\end{equation*}
Note that we added zero to obtain terms II and III in this sum I+II+III+IV of four.
We proceed by estimating each of the four terms individually.
\\
I)~ Concerning term one the identity for $A^+$ corresponding to~(\ref{eq:ident-1}) shows that
\begin{equation*}
\begin{split}
     \Norm{\left(e^{-\tau A}-\1\right) e^{-TA}z_+}
    &=\Norm{\int_0^\tau -A^+ e^{-(s+T)A^+} z_+ \, ds}
   \\
    &\le\Norm{A^+}\cdot\Norm{z_+}\cdot\int_0^\tau e^{-(s+T)\mu}\, ds
   \\
    &\le\lambda_n e^{-T\mu}\frac{1-e^{-\tau\mu}}{\mu}
     \le\tau\lambda_n e^{-T\mu}
\end{split}
\end{equation*}
where the last step uses~(\ref{eq:1-e}) and $\lambda_n$ is the
largest eigenvalue of $A$; see~(\ref{eq:spec-A}). It might be interesting
to see how the in infinite dimensions infinite quantity $\norm{A^+}$
can be avoided, also in III) below; cf.~\cite[p.~954]{weber:2014a}.
\\
II)~For term two we get the estimate
\begin{equation*}
\begin{split}
    &\int_0^T\Norm{e^{-(T+\tau-\sigma)A}\pi_+\left(h(\xi^{T+\tau}(\sigma))-h(\xi^T(\sigma))\right)} d\sigma
   \\
    &\le {c_1}\tau\kappa(\rho) e^{-(T+\tau)\mu}
     \int_0^T e^{\sigma\mu}\, d\sigma
   \\
    &\le {c_1}\tau e^{-\tau\mu} \frac{\kappa(\rho)}{\mu} \left(1-e^{-T\mu}\right)
     \le\frac{{c_1}}{8}\tau
\end{split}
\end{equation*}
where we used~(\ref{eq:xi-xi}) and smallness~(\ref{eq:rho-backward}) of $\rho$.
\\
III)~Term three requires similar techniques as term one and we obtain that
\begin{equation*}
\begin{split}
   &\int_0^T\Norm{\left(e^{-\tau A}-\1\right) e^{-(T-\sigma)A}\pi_+ h(\xi^T(\sigma))}\, d\sigma
   \\
    &\le\int_0^T\int_0^\tau\Norm{A^+ e^{-(T+s-\sigma)A^+}\pi_+ h(\xi^T(\sigma))}\, ds\, d\sigma
   \\
    &\le\lambda_n\kappa(\rho)\int_0^T\int_0^\tau e^{-(T+s-\sigma)\mu}\, ds\, d\sigma
   \\
    &=\lambda_n\kappa(\rho)\cdot\frac{1-e^{-\tau\mu}}{\mu}\cdot\frac{1-e^{-T\mu}}{\mu}
     \le\frac{\lambda_n}{8}\tau
\end{split}
\end{equation*}
where we used once more that $\norm{\xi^T(\sigma)}_X\le 1$ by
Step~1. We also applied~(\ref{eq:1-e}).
\\
IV)~For term four we get the estimate
\begin{equation*}
     \int_T^{T+\tau}\Norm{ e^{-(T+\tau-\sigma)A}\pi_+h(\xi^{T+\tau}(\sigma))} d\sigma
     \le\kappa(\rho)\frac{1-e^{-\tau\mu}}{\mu}
     \le \tau
\end{equation*}
where we used once more the estimates $\kappa(\rho)\le1$ and~(\ref{eq:1-e}).

To summarize, the above estimates show that
\begin{equation}\label{eq:xi-xi-T-tau}
     \Norm{\xi^{T+\tau}(T+\tau)-\xi^T(T)}
     \le\tau \left(c_1+2\lambda_n+1\right)
\end{equation}
for every $\tau\ge0$. This concludes the proof of~(\ref{eq:V-V})
and therefore of Step~4.
\end{proof}

\noindent
{\bf Step 5.} (Exponential $C^0$ convergence)
{\it Given ${z_-}\in S^u_\eps$ and
$z_+\in B^+$, then it holds that
$
     \norm{\Gg^\infty(z_+)
     -\Gg_{z_-}^T(z_+)}
     \le e^{-T\frac{\lambda}{8}}
$
for every $T\ge T_2$.
}

\begin{proof}
Assumption $T\ge T_2$ will be used
in~(\ref{eq:back-traj}).
Fix $z_+\in B^+$ and ${z_-}\in S^u_\eps$ and
consider the fixed point $\xi^T=\xi^T_{{z_-},z_+}$ of
$\Psi_{{z_-},z_+}^T$ on $Z^T$ and the fixed point
$\xi=\xi_{z_+}$ of $\Psi_{z_+}$ on $Z$; cf. proof
of Theorem~\ref{thm:loc-st-mf-thm}.
Since $\Gg^T_{z_-}(z_+)=\xi^T(0)$
by~(\ref{eq:G=xi}) and $\Gg^\infty(z_+)=\xi(0)$ it
remains to estimate the difference $\xi(0)-\xi^T(0)$.
Motivated by~\cite[Fig.~6]{weber:2014a}
the key idea is to suitably decompose the interval $[0,T]$.
The decomposition is illustrated by~\cite[Fig.~5]{weber:2014a}
which is surrounded by ample explication.
Set
$$
     \Norm{\xi-\xi^T}_{C^0_{T/2}}
     :=\sup_{t\in[0,T/2]}\Norm{\xi(t)-\xi^T(t)}.
$$

Pick $t\in[0,\frac{T}{2}]$, use the representation
formulae for $\xi$ and $\xi^T$, the Lipschitz
Lemma~\ref{le:f}, and add several times zero to
get that
\begin{equation*}
\begin{split}
    &\Norm{\xi(t)-\xi^T(t)}
   \\
     &\le
      \int_0^t e^{-(t-\sigma)\mu}
      \Norm{h\circ\xi(\sigma)-h\circ \xi^T(\sigma)}
      \, d\sigma
   \\
     &\quad
      +
      \left(\int_t^{\frac{T}{2}}
      +\int_{\frac{T}{2}}^{\frac{3T}{4}}
      +\int_{\frac{3T}{4}}^T\right)
      e^{(t-\sigma)\mu}
      \Norm{h\circ\xi(\sigma)-h\circ\xi^T(\sigma)}
      \, d\sigma
   \\
     &\quad
      +e^{(t-T)\mu}\Norm{{z_-}}
      +\int_T^\infty e^{(t-\sigma)A^-}
      \Norm{h\circ\xi(\sigma)}
      \, d\sigma
    \\
    &\le   e^{(t-T)\mu}
     +\kappa(\rho)\Norm{\xi-\xi^T}_{C^0_{T/2}}
     \left(
     \int_0^t e^{-(t-\sigma)\mu}\, d\sigma
     +\int_t^{\frac{T}{2}} e^{(t-\sigma)\mu}\, d\sigma
     \right)
   \\
     &\quad
      +\kappa(\rho)
      \int_{\frac{T}{2}}^{\frac{3T}{4}} e^{(t-\sigma)\mu}
      \left(
      \Norm{\xi^T(\sigma)-\phi_\sigma {z_-^T}}
      +\Norm{\phi_\sigma {z_-^T}}
      \right)
      d\sigma
   \\
     &\quad
      +\kappa(\rho)\int_{\frac{3T}{4}}^T
      e^{(t-\sigma)\mu}\Norm{\xi^T(\sigma)}
      \, d\sigma
      +
      \kappa(\rho)\Norm{\xi}_{\exp}
      \int_{\frac{T}{2}}^\infty e^{(t-\sigma)\mu}
      e^{-\sigma\lambda}
      \, d\sigma.
\end{split}
\end{equation*}
The domain of integration $\int_{T/2}^\infty$ in the
last line is not a misprint.

To continue the estimate consider the last three lines. Now we explain how to get to the
corresponding three lines in~(\ref{eq:arrive-at}) below. Concerning \emph{line one} note that
$ e^{(t-T)\mu}\le e^{-T\mu/2}$ since $t\in[0,T/2]$.
Now use the choice of $T_2$ and for the two integrals use~(\ref{eq:int-1}).
In \emph{line two} use that $e^{(t-\sigma)\mu}\le1$, that
$
     \norm{\xi^T(\sigma)
     -\phi_\sigma({z_-^T})}
     \le\rho e^{-\sigma\lambda}
$
by definition of $Z^T$, and that
\begin{equation}\label{eq:back-traj}
     \int_{\frac{T}{2}}^{\frac{3T}{4}}
     \bigl\|\phi_{\sigma-T}({z_-})\bigr\|\, d\sigma
     =
     \int_{\frac{T}{8}}^{\frac{3T}{8}}
     \bigl\|\phi_{-t-\frac{T}{8}}({z_-})\bigr\|\, dt
     \le
     \int_{\frac{T}{8}}^{\frac{3T}{8}} e^{-t\lambda}\, dt.
\end{equation}
Here the identity is by the change of variables $t:=-\sigma+\frac{7}{8}T$.
Set $z_-^*:=z_-^{T/8}$. To see the inequality consider the trajectory $\eta(\tau):=\phi_{\tau-T/8}({z_-})
=\phi_\tau(z_-^*)$ defined for $\tau\in(-\infty,0]$. Observe that $\eta(\tau)\to 0$, as $\tau\to-\infty$,
because ${z_-}$ lies in the descending sphere $S^u_\eps=\p W^u_\eps$ by assumption, thus $z_-^*$ lies
in the backward flow invariant set $\phi_{-T/8}W^u_\eps$. Consequently the whole image of $\eta$ is
contained in $\phi_{-T/8}W^u_\eps\subset W^u_\eps\subset D^-\subset E^-$. Note that $\pi_-\eta(0)=z_-^*$
and consider the unique fixed point $\eta_{z_-^*}$ of the contraction $\Phi_{z_-^*}$
on $Z^-$ defined by~(\ref{eq:Phi}). Backward uniqueness then implies that
$\eta=\eta_{z_-^*}$. Thus $\eta\in Z^-$, so $\norm{\eta(\tau)}\le\rho e^{\tau\lambda}
\le e^{\tau\lambda}$ for every $\tau\le0$ and this proves the inequality~(\ref{eq:back-traj}).
To summarize, line two is bounded from above by
\begin{equation*}
      \kappa(\rho)\cdot\rho
      \left(
     \int_{\frac{T}{2}}^{\frac{3T}{4}} e^{-\sigma\lambda}
     \, d\sigma
     +\int_{\frac{T}{8}}^{\frac{3T}{8}} e^{-t\lambda}\, dt
     \right)
     \le
     \kappa(\rho)\cdot\frac{\rho}{\lambda}
     \left(
     e^{-T\frac{\lambda}{2}}
     +e^{-T\frac{\lambda}{8}}
     \right).
\end{equation*}
Concerning \emph{line three} note that $\norm{\xi^T(\sigma)}\le 1$
since the elements of $Z^T$ take values in $B_1$ and that $\norm{\xi}_{\exp}\le\rho
\le 1$ by definition~(\ref{eq:exp norm}) of $Z$.
Now calculate both integrals (in the second one use $e^{(t-\sigma)\mu}\le1$).

Overall we get that
\begin{equation}\label{eq:arrive-at}
\begin{split}
     \Norm{\xi(t)-\xi^T(t)}
     &\le 
      \frac{1}{8} e^{-T\frac{\mu}{4}}
      +
      \kappa(\rho)
      \left(\frac{1}{\lambda}+\frac{1}{\lambda}\right)
      \Norm{\xi-\xi^T}_{C^0_{T/2}}
   \\
      &\quad
      +
      \frac{2\kappa(\rho)}{\lambda}
      e^{-T\frac{\lambda}{8}}
   \\
      &\quad
      +
      \frac{\kappa(\rho)}{\mu}
      e^{-T\frac{\mu}{4}}
      + 
      \frac{\kappa(\rho)}{\lambda}
      e^{-T\frac{\lambda}{2}}
   \\
    &\le
     \frac{1}{8}
     \Norm{\xi-\xi^T}_{C^0_{T/2}}
     +
     \frac{1}{4} e^{-T\frac{\lambda}{8}}
\end{split}
\end{equation}
for every $t\in [0,T/2]$.
Here the last inequality uses the estimate
$\mu=\lambda+\delta\ge\lambda$ and
smallness~(\ref{eq:rho-backward}) of $\rho$.
Take the supremum over $t\in [0,T/2]$
to obtain that
\begin{equation}\label{eq:xi-eta}
     \Norm{\xi-\xi^T}_{C^0_{T/2}}
     \le \frac{2}{7} e^{-T\frac{\lambda}{8}}.
\end{equation}
Since
$
     \norm{\Gg^\infty(z_+)-\Gg^T_{z_-}(z_+)}=\norm{\xi(0)-\xi^T(0)}
     \le\norm{\xi-\xi^T}_{C^0_{T/2}}
$
this proves Step~5.
\end{proof}

\vspace{.1cm}
\noindent
{\bf Step~6.}
{\rm (Exponential $C^1$ convergence)}
{\it Pick ${z_-}\in S^u_\eps$, $z_+\in B^+$, 
$v\in E^+$, then
$$
     \Norm{d\Gg^T_{z_-}(z_+)v-d\Gg^\infty(z_+)v}
     \le c_* e^{-T\frac{\lambda}{8}}\Norm{v}
$$
for some constant $c_*(\kappa_*,1/\delta,1/\lambda)$
and every $T\ge T_0$ whenever $f$ is of class $C^{2,1}$ near $x$.
}

\begin{proof}
One would expect that the proof of convergence 
of the linearized graph maps should use convergence
of the graph maps themselves.
Indeed estimate~(\ref{eq:xi-eta}) above
is a key ingredient.
Fix $T\ge T_0$, ${z_-}\in S^u_\eps$,
$z_+\in B^+$, and $v\in E^+$.
We proceed in three steps. Step~I and Step~II
are preliminary.

{\bf I.}~For $\tau\ge 0$ small consider the unique
fixed
point $\xi^T_\tau:=\xi^T_{{z_-},z_++\tau v}\in Z^T$
of the contraction $\Psi^T_{{z_-},z_++\tau v}$
defined by~(\ref{eq:Psi}). Set $\xi^T:=\xi^T_0
=\xi^T_{{z_-},z_+}$. By~(\ref{eq:Psi})
\begin{equation}\label{eq:xiT-tau}
\begin{split}
     \xi^T_\tau(t)
    &=
     e^{-tA}\left(z_++\tau v\right)
     +\int_0^t e^{-(t-\sigma)A}\pi_+
     h(\xi^T_\tau(\sigma))\, d\sigma\\
    &\quad
     +e^{-(t-T)A^-}{z_-}-\int_t^T e^{-(t-\sigma)A^-}
     \pi_- h(\xi^T_\tau(\sigma))\, d\sigma
\end{split}
\end{equation}
for every $t\in[0,T]$. By the proof of Step~3 the
composition of maps
$$
     \tau
     \mapsto \xi^T_{{z_-},z_++\tau v}
     \mapsto \xi^T_{{z_-},z_++\tau v} (t)
$$
is of class $C^1$. Therefore the linearization
\begin{equation}\label{eq:def-X_v^T}
     X^T_v(t)
     =X^T_{{z_-},z_+;v}(t)
     :=\left.\frac{d}{d\tau}\right|_{\tau=0}
     \xi^T_{{z_-},z_++\tau v}(t)
\end{equation}
is well defined. For every $t\in[0,T]$
it satisfies the integral equation
\begin{equation}\label{eq:X_v^T}
\begin{split}
     X^T_v(t)
     =e^{-tA} v
    &+\int_0^t e^{-(t-\sigma)A}\pi_+
     \left( dh|_{\xi^T(\sigma)}\circ
     X_v^T(\sigma)\right)\, d\sigma\\
    &-\int_t^T e^{-(t-\sigma)A^-}\pi_-
     \left( dh|_{\xi^T(\sigma)}\circ
     X_v^T(\sigma)\right)\, d\sigma
\end{split}
\end{equation}
and therefore by straighforward calculation
the estimate
\begin{equation}\label{eq:X-exp}
     e^{t\lambda} \Norm{X^T_v(t)}
     \le\Norm{X_v^T}_{\exp}
     \le 2\norm{v}
\end{equation}
for every $t\in[0,T]$.
Use~(\ref{eq:G=xi}) to see that
\begin{equation}\label{eq:X_v^T=dG}
     X_v^T(0)
     :=\left.\tfrac{d}{d\tau}\right|_{\tau=0}
     \xi^T_{{z_-},z_++\tau v}(0)
     =d\Gg^T_{z_-}(z_+) v.
\end{equation}

{\bf II.}~Concerning the local stable manifold
observe the
following. For $\tau\ge0$ small consider the
unique fixed point $\xi_{z_++\tau v}$ of
$\Psi_{z_++\tau v}$ on $Z$ defined
by~(\ref{eq:exp norm}). It satisfies the
integral equation
\begin{equation}\label{eq:eta-tau}
\begin{split}
     \xi_{z_++\tau v} (t)
     =e^{-tA}\left(z_++\tau v\right)
    &+\int_0^t e^{-(t-\sigma)A}\pi_+
     h(\xi_{z_++\tau v} (\sigma))\, d\sigma\\
    &-\int_t^\infty e^{-(t-\sigma)A^-}
     \pi_- h(\xi_{z_++\tau v} (\sigma))\, d\sigma
\end{split}
\end{equation}
for every $t\ge 0$. Hence the linearization
$$
     X_v(t)=X_{z_+;v}(t)
     :=\left.\frac{d}{d\tau}\right|_{\tau=0}
     \xi_{z_++\tau v}(t)
$$
satisfies for every $t\ge 0$ the integral equation
\begin{equation}\label{eq:Y_v}
\begin{split}
     X_v(t)
     =e^{-tA} v
    &+\int_0^t e^{-(t-\sigma)A}\pi_+
     \left( dh|_{\xi_{z_+}(\sigma)}\circ
     X_v(\sigma)\right)\, d\sigma\\
    &-\int_t^\infty e^{-(t-\sigma)A^-}
     \left(\pi_- dh|_{\xi_{z_+}(\sigma)}\circ
     X_v(\sigma)\right)\, d\sigma
\end{split}
\end{equation}
and therefore by straightforward calculation
the estimate
\begin{equation}\label{eq:Y-exp}
     e^{t\lambda} \Norm{X_v(t)}
     \le\Norm{X_v}_{\exp}
     \le 2\norm{v}
\end{equation}
for every $t\ge 0$. Since $\Gg^\infty(z_+)=\xi_{z_+}(0)$ we get that
\begin{equation}\label{eq:Y_v=dG}
     X_v(0)
     :=\left.\tfrac{d}{d\tau}\right|_{\tau=0}
     \xi_{z_++\tau v}(0)
     =d\Gg^\infty(z_+) v.
\end{equation}

{\bf III.}~Abbreviate $\xi^T:=\xi^T_{{z_-},z_+}$ and
$\xi:=\xi_{z_+}$. To estimate the difference
$X^T_v-X_v$ consider the corresponding integral
equations~(\ref{eq:X_v^T}) and~(\ref{eq:Y_v}).
Add zero and apply the Lipschitz Lemma~\ref{le:f}
for $dh$ to obtain that

\begin{equation}\label{eq:df-I-new}
\begin{split}
    &\Norm{dh|_{\xi^T(\sigma)} X_v^T(\sigma)
     -dh|_{\xi(\sigma)} X_v(\sigma)}
   \\
    &=\Norm{\left(dh|_{\xi^T(\sigma)}
     -dh|_{\xi(\sigma)}\right) X_v^T(\sigma)
     +
     dh|_{\xi(\sigma)}\left(X_v^T(\sigma)
     -X_v(\sigma)\right)}
   \\
    &\le
     \kappa_*\Norm{\xi^T(\sigma)-\xi(\sigma)}
     \cdot\Norm{X_v^T(\sigma)}
     +
     \kappa(\rho)\Norm{X_v^T(\sigma)-X_v(\sigma)}
\end{split}
\end{equation}
for every $\sigma\in[0,T]$. Here we applied estimate~(\ref{cor:f}) to get that
\begin{equation}\label{eq:df-II-new}
     \Norm{dh|_{\xi(\sigma)}\circ X_v(\sigma)}
     \le
     \kappa(\rho)\Norm{X_v(\sigma)}.
\end{equation}
Pick $t\in[0,\frac{T}{2}]$ and apply estimates~(\ref{eq:xi-eta}),~(\ref{eq:X-exp}),
and~(\ref{eq:Y-exp}) to obtain
\begin{equation*}
\begin{split}
    &
     \Norm{X_v^T(t)-X_v(t)}
   \\
    &\le
     \int_0^t e^{-(t-\sigma)\mu}\Bigl(\kappa_*
     \underbrace{
       \Norm{\xi^T(\sigma)-\xi(\sigma)}
     }_{\le\frac{2}{7} e^{-T\lambda/8}}
     \underbrace{
       \Norm{X_v^T(\sigma)}
     }_{\le 2 e^{-\sigma\lambda}\norm{v}}
     +\kappa(\rho)
       \Norm{X_v^T(\sigma)-X_v(\sigma)}
     \Bigr) d\sigma
   \\
    &\quad
     +\int_t^{{\frac{T}{2}}} e^{(t-\sigma)\mu}
     \Bigl(\kappa_*\Norm{\xi^T(\sigma)
     -\xi(\sigma)}\Norm{X_v^T(\sigma)}
     +\kappa(\rho)\Norm{X_v^T(\sigma)-X_v(\sigma)}
     \Bigr) d\sigma
   \\
    &\quad
     +\kappa(\rho)\int_{\frac{T}{2}}^T
     e^{(t-\sigma)\mu}\Norm{X_v^T(\sigma)} d\sigma
     +\kappa(\rho)\int_{\frac{T}{2}}^\infty
     e^{(t-\sigma)\mu}
     \underbrace{
       \Norm{X_v(\sigma)}
     }_{\le 2 e^{-\sigma\lambda}\norm{v}}
     d\sigma
   \\
    &\le
     \frac{4}{7}\kappa_* e^{-T\frac{\lambda}{8}}\Norm{v}
     \left(
     \int_0^t e^{-(t-\sigma)\mu} e^{-\sigma\lambda}
     \, d\sigma
     +
     \int_t^{\frac{T}{2}} e^{(t-\sigma)\mu} e^{-\sigma\lambda}
     \, d\sigma
     \right)
   \\
    &\quad
     +\kappa(\rho)\Norm{X_v^T-Y_v}_{C^0_{T/2}}
     \left(
     \int_0^t e^{-(t-\sigma)\mu}\, d\sigma
     +
     \int_t^{\frac{T}{2}} e^{(t-\sigma)\mu}\, d\sigma
     \right)
   \\
    &\quad
     +4\kappa(\rho)\Norm{v}
     \int_{{\frac{T}{2}}}^\infty
     e^{(t-\sigma)\mu} e^{-\sigma\lambda}
     \, d\sigma
   \\
    &\le 
     \left(\kappa_*
     \left(\frac{e^{-t\lambda}}{\delta}
     +\frac{e^{-t\lambda}}{\lambda}\right)
     +\frac{4\kappa(\rho)}{\lambda}\right)
     \Norm{v} e^{-T\frac{\lambda}{8}}
     +\kappa(\rho)\left(\frac{1}{\lambda}
     +\frac{1}{\lambda}\right)
     \Norm{X_v^T-X_v}_{C^0_{T/2}}
   \\
    &\le
     \frac{c_*}{2}
     \Norm{v} e^{-T\frac{\lambda}{8}}
     +\frac{1}{16}
     \Norm{X_v^T-X_v}_{C^0_{T/2}}
\end{split}
\end{equation*}
where $c_*:=2\kappa_*(\frac{1}{\delta}+\frac{1}{\lambda})+1/4$.
In inequality three we calculated the integrals
and estimated $\frac{1}{\mu}\le \frac{1}{\lambda}$.
The final inequality is by
smallness~(\ref{eq:rho-backward}) of $\rho$.
Now take the supremum over $t\in[0,T/2]$
and use the resulting estimate to continue
\begin{equation*}
\begin{split}
     \Norm{d\Gg^T_{z_-}(z_+)v-d\Gg^\infty(z_+)v}
   &
     =\Norm{X_v^T(0)-X_v(0)}
   \\
   &
     \le\Norm{X_v^T-X_v}_{C^0_{T/2}}.
\end{split}
\end{equation*}
This concludes the proof of Step~6 and Theorem~\ref{thm:backward-lambda-Lemma}.
\end{proof}

\section{Invariant stable foliations and induced flow}\label{sec:inv-fol}
Given $x\in\Crit f$ non-degenerate, set $c:=f(x)$.
Consider the local model on $V=W^u_{2\varsigma}\times W^s_{2\varsigma}\subset T_xM$ provided
by~(H3) in Definition~\ref{def:loc-coord-lambda} and recall that the assertions of the backward
$\lambda$-Lemma, Theorem~\ref{thm:backward-lambda-Lemma}, hold true
for each $\eps\in(0,\varsigma)$.

\subsection*{Invariant stable foliations}
\begin{theorem}[Invariant stable foliation]\label{thm:inv-fol}
Given $x\in\Crit f$ non-degenerate, set $k=\IND(x)$ and $c=f(x)$. Then
for every sufficiently small $\eps>0$ the following is true. Pick a tubular neighborhood
$\Dd(x)$ (associated to a radius $\varkappa$ normal disk bundle) over the descending
sphere $S^u_\eps(x)$ in the level hypersurface $\{f=c-\eps\}$. Denote the fiber
over $q\in S^u_\eps(x)$ by $\Dd_q(x)$; see Figure~\ref{fig:fig-Conley-pair}.
Then the following holds for every sufficiently large $\tau>1$.
\begin{description}
\item[(Foliation)]
  The subset $N_x^s$ of $M$ is compact and contains no critical points 
  except $x$. Moreover, it carries the structure of a
  continuous codimension $k$ foliation\footnote{
    For the precise degrees of smoothness away from $W^s_\eps(x)$ see
    Theorem~\ref{thm:backward-lambda-Lemma}.
    }
  whose leaves are parametrized by the $k$ disk $\varphi_{-\tau} W^u_\eps(x)=N^s_x\cap W^u(x)$.
  The leaf $N_x^s(x)$ over $x$ is the
  ascending disk $W^s_\eps(x)$. The other leaves are the codimension $k$ disks given by
  \begin{equation*}
     N_x^s(q^T)
     =\left({\varphi_T}^{-1}\Dd_q (x)\cap
     \{f\le c+\eps\}\right)_{q^T}
     ,\qquad
     q^T:=\varphi_{-T}q,
  \end{equation*}
  whenever $T\ge\tau$ and $q\in S^u_\eps(x)$.
\item[(Invariance)]
  Leaves and semi-flow are compatible in the sense that
  \begin{equation*}
     p\in N_x^s(q^T)
     \quad\Rightarrow\quad
     \varphi_\sigma p\in N_x^s(\varphi_\sigma q^T)
     \quad \forall \sigma\in[0,T-\tau).
  \end{equation*}
\item[(Contraction onto ascending disk)]
  The leaves converge uniformly to the ascending disk in the sense that
  \begin{equation}\label{eq:unif-exp-dist}
     \dist\left(N_x^s(q^T),W^s_\eps(x)\right)
     \le e^{-T\frac{\lambda}{8}}
  \end{equation}
  for all $T\ge\tau$ and $q\in S^u_\eps(x)$.\footnote{
    Here $\dist$ denotes the infimum over all Riemannian distances of any two points
    and the constant $\lambda\in(0,d)$ has been fixed in~(H2).
    }
  If $U$ is a neighborhood of $W^s_\eps(x)$ in $M$, then
  $N_x^s(\eps,\tau_*)\subset U$ for some constant $\tau_*$.
\item[(Shrink to critical point)]
  Assume $U$ is a neighborhood of $x$ in $M$. Then there are constants $\eps_*$ and
  $\tau_*$ such that $N_x^s(\eps_*,\tau_*)\subset U$.
\end{description}
\end{theorem}

Theorem~\ref{thm:inv-fol} is illustrated by Figure~\ref{fig:fig-N}.
Theorem~\ref{thm:backward-lambda-Lemma} asserts that some
neighborhood of $x$ is the union of images of $(n-k)$-disks.
The first step (Corollary~\ref{cor:inv-fol}) is to show that these
images are disjoint and compatible with the flow.
The actual proof of Theorem~\ref{thm:inv-fol} coincides word by
word with the infinite dimensional version~\cite[Thm.~C]{weber:2014c}.
To see this coincidence set $R=\frac{\rho}{4}$
and impose on $\eps>0$ the condition $W^s_{\eps}\subset B^+_R$.
Note that $B^+_{2R}=B^+\subset D^+$.

\subsubsection*{Local non-intrinsic foliation}\label{sec:local-fol}
\begin{definition}\label{def:no-return}
A set $\Dd$ has the {\bf no return property}
with respect to a forward flow $\phi$ if $\Dd\cap{\phi_t}^{-1}\Dd=\emptyset$, equivalently
$\phi_t \Dd\cap \Dd=\emptyset$, whenever $t>0$.
\end{definition}

Clearly any subset of a \emph{level} set has the no return property with respect to a \emph{gradient} semi-flow.
Recall that $\Gg^\infty$ denotes the stable manifold graph map provided by
Theorem~\ref{thm:loc-st-mf-thm} and $\Gg$ is the graph map which
appears in the backward $\lambda$-Lemma, Theorem~\ref{thm:backward-lambda-Lemma}.

\begin{corollary}[to the Backward $\lambda$-Lemma]\label{cor:inv-fol}
Under the assumptions of Theorem~\ref{thm:backward-lambda-Lemma}
and with the additional assumption that $(D,\phi)$ has the no return property the
following is true. The subset
$$
     F=F^{\eps,T_0}
     :=\left(\im\Gg\cup\im\Gg^\infty\right)
     \subset D^-\times B^+\subset V
$$
of $X=T_x M$ carries the structure of a
continuous foliation of codimension $k=\IND(x)$.\footnote{
  For the definition of foliation see e.g.~\cite{lawson:1974a}
  or~\cite[Sec.~4.1]{pesin:2004a} and
  for the precise degrees of smoothness away from $W^s_\eps(x)$ see
  Theorem~\ref{thm:backward-lambda-Lemma}.
  }
The leaf over the singularity $0$ is given by the subset
$F(0):=\Gg^\infty(B^+)$ of the ascending disk $W^s_{2\varsigma}=D^+$
and the leaf over the points
$$
     \alpha^T:=\phi_{-T}{\alpha}=\Gg^T_{\alpha}(0),\quad
     T\ge T_0,\quad {\alpha}\in S^u_\eps,
$$
are given by the graphs $F({\alpha^T}):=\Gg^T_{\alpha}(B^+)$.
Leaves and flow are compatible in the sense that
$$
     z\in F({\alpha^T})
     \quad\Rightarrow\quad
     \phi_\sigma z\in F(\phi_\sigma{\alpha^T})
$$
whenever the flow trajectory from $z$ to 
$\phi_\sigma z$ remains inside $F$.
\end{corollary}

\begin{proof}[Proof of Corollary~\ref{cor:inv-fol}]
By Theorem~\ref{thm:backward-lambda-Lemma}
it remains to check that the sets $F({\alpha^T})$
and $F(\beta^S)$ are disjoint whenever
${\alpha^T}\not=\beta^S$, that is
$(T,{\alpha})\not= (S,\beta)$.
Assume by contradiction that 
$\Gg^T_{\alpha}(z_+)=\Gg^S_\beta(z_+)=:z$ for some
$z_+\in B^+$. Then by~(\ref{eq:G=xi}) the
point $z$ is the initial value of a flow
trajectory $\xi^T$ ending at time $T$ on the fiber
$\Dd_{\alpha}$ and also of a flow trajectory $\xi^S$
ending at time $S$ on $\Dd_{\beta}$.
By uniqueness of the solution to the Cauchy 
problem with initial value $z$
the two trajectories coincide until time $\min\{T,S\}$.
If $T=S$, then ${\alpha}=\beta$ and we are done.
Now assume without loss of generality that $T<S$, 
otherwise rename. Hence $\xi^S$ meets $\Dd_{\alpha}$
at time $T$ and $\Dd_\beta$ at the later time $S$.
But this contradicts the no return property of $D$.

We prove compatibility of leaves and flow. The
fixed point $0$ is flow invariant. Its
neighborhood $F(0)$ in the ascending disk is
trivially flow invariant in the required sense,
namely up to leaving $F(0)$.
Pick $z\in F({\alpha^T}):=\Gg^T_{\alpha}(B^+)$.
By~(\ref{eq:G=xi}) the point $z$ is the initial value
of a flow trajectory $\xi^T$ ending at time $T$ on
the fiber $\Dd_{\alpha}$.
Assume the image $\xi^T([0,T])=\phi_{[0,T]} z$ is
contained in $F:=\im\Gg\cup\im\Gg^\infty$ and pick
$\sigma\in[0,T-T_0]$. This implies that
$z_+:=\pi_+\phi_\sigma z\in B^+$. The
flow line $\phi_{[0,T-\sigma]}\phi_\sigma z$ runs
from $\phi_\sigma z$ to $\phi_T z\in \Dd_{\alpha}$.
Thus by uniqueness this flow line coincides with the
fixed point $\xi_{{\alpha},z_+}^{T-\sigma}$ of the strict
contraction $\Psi_{{\alpha},z_+}^{T-\sigma}$. But
$\phi_\sigma z=\xi_{{\alpha},z_+}^{T-\sigma}(0)$ is equal
to $\Gg^{T-\sigma}_{\alpha}(z_+)$ again
by~(\ref{eq:G=xi}). Thus
$$
     \phi_\sigma z\in
     \Gg^{T-\sigma}_{\alpha}(B^+)
     =:F({\alpha}^{T-\sigma})
     =F(\phi_\sigma{\alpha^T})
$$
by definition of $F$ and ${\alpha}^{T-\sigma}$.
\end{proof}

\subsection*{Induced flow -- Dynamical thickening}

\begin{theorem}[Strong deformation retract]\label{thm:deformation-retract}
Consider a pair of spaces $(N^s_x,L^s_x)$ as defined
by~(\ref{eq:Conley-set-NEW}--\ref{eq:L_x}).
Then the following is true for all pair parameters $\eps>0$ sufficiently small
and $\tau>1$ sufficiently large. Firstly, the pair strongly
deformation retracts onto its part in the unstable manifold, that is onto
$$
     (D,A)=(D^s_x,A^s_x):=(N^s_x\cap W^u(x),L^s_x\cap W^u(x)) .
$$
Moreover, this pair consists of the closed disk $D=\varphi_{-\tau} W^u_\eps(x)$
whose dimension $k$ is the Morse index of $x$ and an annulus $A$ which arises
by removing from $D$ the smaller open disk $\INT \varphi_{-2\tau} W^u_\eps(x)$;
see Figure~\ref{fig:fig-Conley-pair}.
\end{theorem}

\begin{corollary}\label{cor:conley-index-N-L}
For any pair of spaces in Theorem~\ref{thm:deformation-retract} it holds that
\begin{equation}\label{eq:N-L}
   \Ho_\ell(N^s_x,L^s_x)\cong
   \begin{cases}
       \Z&,\ell=k:=\IND(x),
       \\
       0&,\text{otherwise.}
   \end{cases}
\end{equation}
\end{corollary}

To construct a deformation to prove Theorem~\ref{thm:deformation-retract}
there is the immediate temptation to use the already present forward
flow $\varphi_t$. Unfortunately and obviously, this only works along the stable manifold
where indeed the flow moves any point into the unstable manifold, as time $t\to\infty$.
Along the complement of the stable manifold this does not work at all.\footnote{
  This is known as \emph{discontinuity of the flow end point map on
  unstable manifolds} and obstructs simple \emph{geometric} proofs of a
  number of desirable results such as the one that the unstable
  manifolds of a Morse-Smale gradient flow on a closed manifold
  naturally provide a CW decomposition; see e.g.~{\cite{bott:1988a,banyaga:2004a,nicolaescu:2011a,burghelea:2011a-arXiv,2011arXiv1102.2838Q}}.
  }
\begin{figure}
  \centering
  \includegraphics{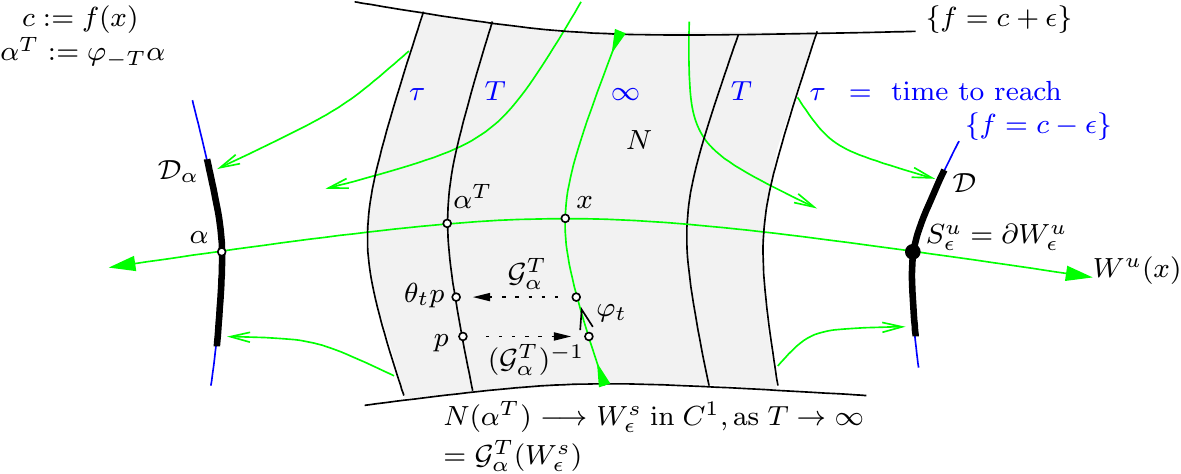}
  \caption{Dynamical thickening $(N,\theta)=(N^s_x,\theta^s_x)$
                 of ascending disk $W^s_\eps(x)$
                }
  \label{fig:fig-N}
\end{figure}
However, given the foliation in Theorem~\ref{thm:inv-fol} of $N_x^s$ in terms of \emph{graphs},
it is then a natural idea to turn this foliation into a \emph{dynamical foliation}
in the sense that each leaf will be endowed with its own flow. The natural candidate
for these flows is our given flow $\varphi_t$ on the ascending disk
$W^s_\eps(x)=N^s_x\cap W^s(x)$ transported to any leaf via conjugation by the corresponding
graph map; see Figures~\ref{fig:fig-N}\footnote{
  In Figure~\ref{fig:fig-N} we use \emph{simultaneously} global and
  local coordinate notation for illustration.
  } and~\ref{fig:fig-induced-flow}.
This way the set $N_x^s$ turns into a disjoint union of copies of the dynamical
system $(W^s_\eps(x),\varphi)$. So one could call this procedure a {\bf dynamical thickening} of
the stable manifold; see~\cite{weber:2014e} for an application to a classical theorem.

\begin{proof}[Proof of Theorem~\ref{thm:deformation-retract}]
Throughout we work in the local model provided by Definition~\ref{def:loc-coord-lambda}.
In particular, we will use these notation conventions in our construction of a strong
deformation retraction $\theta$ of $N$ onto its part $D$ in the unstable manifold.
As pointed out above on the stable manifold the forward flow $\{\phi_t\}_{t\in[0,\infty]}$
itself does the job. Indeed $\phi_\infty$ pushes the whole leaf $N(0)$, that is the ascending disk
$W^s_\eps$ by Theorem~\ref{thm:inv-fol}, into the origin -- which lies in the unstable manifold.
Since $\phi_t$ restricted to the origin is the identity, the origin is a strong deformation retract of $N(0)$.
If the Morse index $k$ is zero, then $N=N(0)$ and we are done; similarly for $k=n$.
Assume from now on $1\le k\le n-1$. The main idea is to use the graph
maps $\Gg^\infty$ and $G^T_\alpha$ provided by
Theorems~\ref{thm:loc-st-mf-thm} and~\ref{thm:backward-lambda-Lemma}, respectively,
and their left inverse $\pi_+$ to extend the good retraction properties of  $\phi_t$ on the
ascending disk $N(0)$ to all the other leaves $N(\alpha^T)$ provided
by Theorem~\ref{thm:inv-fol}.

\begin{definition}[Induced semi-flow -- Dynamical thickening]
\label{def:induced-semi-flow}
By Theorem~\ref{thm:inv-fol} each $z\in N=N^{\eps,\tau}$ lies on a leaf
$N(\alpha^T)$ for some $T\ge\tau$ and some $\alpha$ in the ascending
disk $S^u_\eps$ and where $\alpha^T:=\phi_{-T}\alpha$.
Then the continuous map $\theta:[0,\infty]\times N\to X$ defined by
\begin{equation}\label{eq:str-def-retract}
     \theta_t z
     :=\Gg_\alpha^T\circ\pi_+\circ \phi_t\circ\Gg^\infty\circ \pi_+(z)
     =\Gg_\alpha^T\circ\phi_t\circ\pi_+(z),
\end{equation}
is called the {\bf induced semi-flow on \boldmath $N$}; see Figure~\ref{fig:fig-induced-flow}.
It is of class $C^r$ away from the stable manifold. Here~(\ref{eq:str-def-retract})
simplifies as $\Gg^\infty=(0,id)$ and $\pi_+|_{E^+}=\1$.
\end{definition}

\begin{figure}
  \centering
  \includegraphics{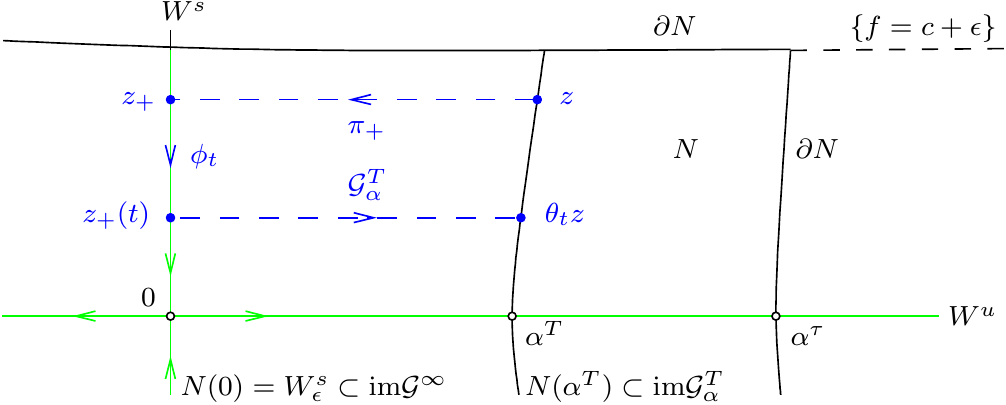}
  \caption{Leaf preserving semi-flow $\theta_t$
          }
  \label{fig:fig-induced-flow}
\end{figure}

Observe that while $\theta$ apriori only takes values in the image
$F\supset N$ of the graph maps,
it does preserve the \emph{leaves of $F$}; see Corollary~\ref{cor:inv-fol}.
Now continuity on $[0,\infty)\times N$ follows from continuity of the maps involved.
Given $z\in N(\alpha^T)$, set $z_+:=\pi_+ z$. Since $z_+$ lies in the
stable manifold by construction of our ''flat'' coordinates,
we obtain that $z_+(t):=\phi_t z_+\to 0$, as $t\to\infty$. So firstly the limit
\begin{equation*}
     \theta_\infty z
     :=\lim_{t\to\infty} \theta_t z
     =\Gg^T_\alpha\lim_{t\to\infty}z_+(t)
     =\Gg^T_\alpha(0)
     =\alpha^T
\end{equation*}
exists and lies in the unstable manifold indeed. Here we used continuity of $\Gg^T_\alpha$
and the final identity holds by Theorem~\ref{thm:backward-lambda-Lemma}.
Secondly $\theta_t z\to\theta_\infty z$, as $t\to\infty$. This shows, firstly, that the map
\begin{equation}\label{eq:strong-def-retract}
     \theta_\infty:N\to D
     ,\qquad
     D=\phi_{-\tau} W^u_\eps
     \cong\{0\}\cup\left([\tau,\infty)\times S^u_\eps\right),
\end{equation}
is a retraction and, secondly, that the map $\theta$ extends continuously to $[0,\infty]\times N$.
The fact that the origin is a fixed point of $\pi_+$ and $\phi_t$ implies that
\begin{equation*}
     \theta_t\alpha^T=\Gg^T_\alpha\circ\phi_t\circ\pi_+(0)
     =\Gg^T_\alpha(0)=\alpha^T.
\end{equation*}

Hence $\theta_t|_D=id_D$, for every $t\in[0,\infty]$.
It remains to show that $\theta_t$ actually preserves $N$.
It suffices to show that $\theta_t$ preserves each leaf of the foliation
$$
     N=N(0)\cup
     \bigcup_{\substack{T\ge\tau\\\alpha\in S^u_\eps}}
     N(\alpha^T).
$$
In contrast to the infinite dimensional case~\cite{weber:2014c} a
simple compactness argument will do.
Note that any leaf, other than $N(0)=W^s_\eps$, is of the form
$$
     N(\alpha^T)
     =\Gg^T_\alpha(B^+)\cap\{f\le c+\eps\}
     ,\qquad
     \p N(\alpha^T)=\Gg^T_\alpha(B^+)\cap\{f=c+\eps\},
$$
whereas $N(0)=\Gg^\infty(B^+)\cap\{f\le c+\eps\}=W^s_\eps$ and $\p N(0)=S^s_\eps$.
Now since the boundary of $N(\alpha^T)$ lies in a level set of $f$
and $\theta_s$ preserves the graph of $\Gg^T_\alpha$ we only need to
show that there is a constant $-\mu<0$ such that
\begin{equation}\label{eq:theta-mu}
     \left.\frac{d}{dt}\right|_{t=0}f(\theta_t z)
     \le-\mu<0,\qquad
     \forall z\in\p N(\alpha^T),\quad\forall\alpha^T\in D.
\end{equation}
This means that the $\theta$ flow is inward pointing along the
boundary of each leaf and then we are done.
But~(\ref{eq:theta-mu}) holds true for some constant, call it $-2\mu$, along
the (compact) boundary
$S^s_\eps$ of the leaf $N(0)$, just because $\theta_t=\phi_t$ on $N(0)$ and $f$
strictly decreases along its downward gradient flow -- unless there is a
critical point which it isn't on $S^s_\eps$.
Compactness of the leaf space $D$ and of the boundary of each leaf,
together with continuity of the maps whose composition is $\frac{d}{dt}\theta_t$,
then implies that~(\ref{eq:theta-mu}) holds true for all nearby
leaves. To restrict to nearby leaves just fix $\eps>0$
sufficiently small and $\tau>1$ sufficiently large.
\end{proof}

\subsubsection*{Global foliation}
To extend the foliation of the neighborhood $N_x^s$ of $x$ all along
the stable manifold $W^s(x)$, that is in the backward time direction,
according to Convention~\ref{con:back-flow}, that is without using the backward flow,
is an interesting open problem -- even more so in heat flow situations
such as~\cite{weber:2014c}.



\bibliographystyle{alpha}
\bibliography{ISTIME-5.bbl}{}

%


%
\end{document}

%% file: latex-shortcuts.tex
\newtheorem{theorem}{Theorem}[section]
\newtheorem{corollary}[theorem]{Corollary}

\newtheorem{lemma}[theorem]{Lemma}
\newtheorem{proposition}[theorem]{Proposition}

\theoremstyle{definition}
\newtheorem{definition}[theorem]{Definition}

\newtheorem{convention}[theorem]{Convention}
\newtheorem{remark}[theorem]{Remark}

\theoremstyle{remark}

\newcommand{\D}{{\mathbb{D}}}

\newcommand{\R}{{\mathbb{R}}}

\newcommand{\Z}{{\mathbb{Z}}}
\newcommand{\Dd}{{\mathcal{D}}}

\newcommand{\Ff}{{\mathcal{F}}}
\newcommand{\Gg}{{\mathcal{G}}}   

\newcommand{\Oo}{{\mathcal{O}}}

\newcommand{\im}{{\rm im\, }}      

\newcommand{\dist}{{\rm dist}}     
\newcommand{\INT}{{\rm int\, }}       
\newcommand{\IND}{{\rm ind}}       
 
\newcommand{\Crit}{{\rm Crit}}        
\newcommand{\Ho}{{\rm H}}             
\newcommand{\Hess}{\mathrm{Hess}}        

\newcommand{\norm}{{\rm norm}}

\newcommand{\eps}{{\varepsilon}}

\def\NABLA#1{{\mathop{\nabla\kern-.5ex\lower1ex\hbox{$#1$}}}}
\def\Nabla#1{\nabla\kern-.5ex{}_{#1}}
\def\Tabla#1{\Tilde\nabla\kern-.5ex{}_{#1}}
\def\abs#1{\mathopen|#1\mathclose|}   
\def\Abs#1{\left|#1\right|}            
\def\norm#1{\mathopen\|#1\mathclose\|}
\def\Norm#1{\left\|#1\right\|}

\renewcommand{\Tilde}{\widetilde}

\newcommand{\p}{{\partial}}

\renewcommand{\1}{{{\mathchoice {\rm 1\mskip-4mu l} {\rm 1\mskip-4mu l}
{\rm 1\mskip-4.5mu l} {\rm 1\mskip-5mu l}}}}